\documentclass[12pt,reqno]{amsart}
\usepackage{amsmath,amsfonts,amsthm,amssymb,color,mathrsfs}
\usepackage [latin1]{inputenc}

\usepackage{hyperref}
\hypersetup{
     colorlinks   = true,
     citecolor    = blue,
     linkcolor    = blue
}

\usepackage[T1]{fontenc}
\usepackage{pdfsync}

\usepackage[left=1.0in, right=1.0in, top=1.1in,bottom=1.1in]{geometry}
\definecolor{dg}{rgb}{0, 0.5, 0}

\newcommand{\ch}{{{\mathcal{H}}}}
\newcommand{\tch}{{\tilde{\mathcal H}}}

\newcommand{\bn}{\mathbf{n}}

\newcommand{\bd}{\mathbf{D}}
\newcommand{\bj}{\mathbf{J}}

\newcommand{\bX}{\mathbf{X}}

\newcommand{\der}{\delta}

\newcommand{\id}{\mbox{Id}}

\newcommand{\ot}{[0,t]}

\newcommand{\1}{{\bf 1}}

\newcommand{\lp}{\left(}
\newcommand{\rp}{\right)}
\newcommand{\lc}{\left[}
\newcommand{\rc}{\right]}
\newcommand{\lcl}{\left\{}
\newcommand{\rcl}{\right\}}
\newcommand{\lln}{\left|}
\newcommand{\rrn}{\right|}
\newcommand{\lla}{\left\langle}
\newcommand{\rra}{\right\rangle}

\newcommand{\al}{\alpha}
\newcommand{\ep}{\varepsilon}

\newcommand{\ga}{\gamma}
\newcommand{\ka}{\kappa}
\newcommand{\la}{\lambda}

\newcommand{\om}{\omega}
\newcommand{\oom}{\Omega}

\newcommand{\si}{\sigma}

\newcommand{\eps}{\varepsilon}

\newcommand{\cac}{{\mathcal C}}

\newcommand{\ce}{{\mathcal E}}
\newcommand{\cf}{{\mathcal F}}

\newcommand{\crr}{{\mathcal R}}

\newcommand{\beq}{\begin{equation}}
\newcommand{\eeq}{\end{equation}}
\newcommand{\bea}{\begin{eqnarray}}
\newcommand{\eea}{\end{eqnarray}}
\newcommand{\beas}{\begin{eqnarray*}}
\newcommand{\eeas}{\end{eqnarray*}}

\def\me{{\mathbb  E}}

\def\md{{\mathbb D}}
\def\mr{{\mathbb  R}}

\def\mp{{\mathbb  P}}

\newcommand{\EE}{{\mathbb E}}

\newcommand{\PP}{{\mathbb P}}

\newcommand{\R}{{\mathbb R}}
\newcommand{\Z}{{\mathbb Z}}
\newtheorem{theorem}{Theorem}[section]

\newtheorem{corollary}[theorem]{Corollary}

\newtheorem{definition}[theorem]{Definition}

\newtheorem{hypothesis}[theorem]{Hypothesis}
\newtheorem{lemma}[theorem]{Lemma}

\newtheorem{proposition}[theorem]{Proposition}
\theoremstyle{remark}
\newtheorem{remark}[theorem]{Remark}
\theoremstyle{remark}
\newtheorem{example}[theorem]{Example}
\theoremstyle{remark}

\newtheorem{foo}[theorem]{Remarks}

\newcommand{\var}[1]{{\rm Var}\left(#1\right)}

\title[Density bounds for Gaussian RDEs]{Density bounds for solutions  to differential equations driven by Gaussian rough paths}

\author[B. Gess \and C. Ouyang \and S. Tindel]
{Benjamin Gess \and Cheng Ouyang \and Samy Tindel}

\thanks{ %B. Gess acknowledges support from the European Union program  FP7-PEOPLE-2012-CIG under grant agreement 333938. 
C. Ouyang' research is supported in part by Simons grant \#355480.   S. Tindel is supported in part by
NSF Grant DMS 0907326. }



\date{\today}

\address{Benjamin Gess, Max Planck Institute for Mathematics in the Sciences, Leipzig \& University of Bielefeld, Germany}
\email{benjamin.gess@gmail.com}

\address{Cheng Ouyang, Dept. Mathematics, Statistics and Computer Science, University of Illinois at Chicago, 851 S. Morgan St., Chicago, IL 60607, USA.}
\email{couyang@math.uic.edu}

\address{Samy Tindel, Dept. Mathematics, Purdue University, 150 N. University St., West La\-fa\-ye\-tte, IN 47907-2067, USA.}
\email{stindel@math.purdue.edu}

\begin{document}

\begin{abstract}
We consider finite dimensional rough differential equations driven by centered Gaussian processes. Combining Malliavin calculus, rough paths techniques and interpolation inequalities, we establish upper bounds on the density of the corresponding solution for any fixed time $t>0$. In addition, we provide Varadhan estimates for the asymptotic behavior of the density for small noise. The emphasis is on working with general Gaussian processes with covariance function satisfying suitable abstract, checkable conditions.
\end{abstract}

\maketitle

{
\hypersetup{linkcolor=black}
 \tableofcontents 
}

\section{Introduction}

Let $p_{t}$ be the density of the solution $Y_{t}^{z}$ to a  stochastic differential equation 
\begin{equation}\label{eq:diffusion-process} 
Y^{z}_t =z +\int_0^t V_0 (Y^z_s)ds+
\sum_{i=1}^d \int_0^t V_i (Y^{z}_s) dB^i_s,
\end{equation}
driven by a $d$-dimensional Brownian motion $B$, where $z\in\R^{n}$ is a given initial condition and $V_0,\ldots,V_d$ are smooth vector fields on $\R^n$. In this classical setting and under non-degeneracy conditions on the vector fields $V_0,\ldots,V_d$, it is a well-know fact that $p_{t}$ behaves like a Gaussian density. Such results can be obtained by considering the PDE governing $p_{t}$, which relies on the Markovian nature of \eqref{eq:diffusion-process}. Alternatively, due to the celebrated proof of H\"ormander's theorem by Malliavin \cite{Ma}, more probabilistic tools have been used in order to analyze laws of solutions to stochastic differential equations. This kind of technology has paved the way to the extension of such results to a much broader class of differential equations, such as delayed equations \cite{BM,FRS} and stochastic PDE (see e.g \cite{BP,NQ,RS} among many others).

While the above equation \eqref{eq:diffusion-process} is restricted to Brownian noise, Terry Lyons' theory of rough paths allows to study more general stochastic differential equations of the type
\begin{equation}\label{eq:rde-intro} 
Z^{z}_t =z +\int_0^t V_0 (Z^z_s)ds+
\sum_{i=1}^d \int_0^t V_i (Z^{z}_s) dX^i_s,
\end{equation}
driven by general $p$-rough paths $X$. Among the processes $X$ to which the abstract theory of rough paths can be applied, fractional Brownian motion has attracted a lot of attention in recent years. Indeed, based on several recent works in this direction, the law of the solution to \eqref{eq:rde-intro} driven by fractional Brownian motion is now fairly well understood. Important results in this direction include the existence of a density, smoothness results, Gaussian bounds, short time asymptotics, invariant measures, hitting probabilities and the existence of local times {(see \cite{BH, CF, CHLT, BOT, BO, Inahama2, H, BNOT, LO2} and the references therein).}

Much less is known for differential equations \eqref{eq:rde-intro} driven by general Gaussian processes. This is in contrast to the theory of rough paths, which covers a lot more than fractional Brownian motion. In fact, the existence of a rough path lift for Gaussian processes is naturally related to the existence of 2-d Young type integrals for the covariance function $R$, as highlighted in \cite{FV-bk} and improved in \cite{FGGR13} based on mixed variations of $R$. In addition, in \cite{FGGR13} the applicability to a wide variety of Gaussian processes, such as Gaussian random Fourier series and bifractional Brownian motions is shown, hence allowing to give a meaning and solve equations of the form \eqref{eq:rde-intro} in this general framework. Further studies of differential equations driven by general Gaussian processes include H\"ormander type theorems under general local non-determinism type conditions on the covariance $R$ (see \cite{CHLT}).

The current article is a further development towards a more complete description of differential equations \eqref{eq:rde-intro} driven by general Gaussian processes. More precisely, we consider \eqref{eq:rde-intro} driven by a Gaussian process $X$ satisfying appropriate general, checkable conditions. Assuming ellipticity conditions on the vector fields $V_0,\ldots,V_d$ and natural conditions on the covariance $R$, we prove that the density of $Z_{t}$ admits a sub Gaussian upper bound (Theorem \ref{thm:upper-bnd-density} below). Moreover, we show in Theorem \ref{th: main result} below that the density satisfies Varadhan type estimates for small noise.  The proof of the above results is based on stochastic analysis tools and, more specifically, on an integration by parts formula which gives an exact expression for the density function in terms of the Malliavin derivatives and the Malliavin matrix of $Z$. Thus, a large part of the paper is devoted to obtaining precise estimates for the Malliavin derivative and Malliavin matrix.

The assumptions on the driving Gaussian process are quite standard in the rough paths literature and can be divided into the following two groups:

\noindent
\emph{(i)}
Similarly to \cite{FGGR13}, we assume that the covariance function $R$ has finite mixed $(1,\rho)$-variation for some $\rho\in[1,2)$ in order to ensure that the driving process $X$ admits a rough path lift and complementary Young regularity is satisfied.

\noindent
\emph{(ii)}
In order to analyze the inverse of the Malliavin matrix of the solution $Z$, we rely on interpolation inequalities for the Cameron-Martin space related to $X$ (see Proposition \ref{th: interpolation} below), which in turn rely on monotonicity conditions on the increments of the covariance $R$ (see Hypotheses \ref{hyp:correlation-increments-X} below) and so-called non-determinism conditions (Hypothesis \ref{assumption1} below), which have already been used in \cite{CHLT}.

The rest of the paper is organized as follows. In Section 2, we provide some basic tools from Malliavin calculus and rough path theory that will be needed later. We also set up corresponding notations in this section. Section 3 is devoted to obtaining the upper bound of the density, while Section 4 focuses on Varadhan estimates. Finally, in Section 5, we provide several examples of Gaussian rough paths that satisfy the general assumptions supposed in the main body of this work.

\textbf{Notations:} Throughout this paper, unless specified otherwise, we denote Euclidean norms by $|\cdot|$. The space of $\mr^n$-valued $\gamma$-H\"older continuous functions defined on $[0,T]$ will be denoted by $\cac^\ga([0,T], \mr^n)$ and $\cac^\ga$ for short. For a function $g\in\cac^\ga([0,T],\mr^n)$ and $0\le s<t\le T$, we shall consider the semi-norms
\begin{equation}\label{eq:def-holder-norms}
\|g\|_{\ga; [s,t]}:=\sup_{s\le u<v\le t}\frac{|g_v-g_u|}{|v-u|^{\ga}}.
\end{equation}
%The semi-norm $\|g\|_{0,1,\ga}$ will simply be denoted by $\|g\|_{\ga}$.
Generic universal constants will be denoted by $c,C$ independently of their exact values.

\section{Preliminary material}\label{sec:preliminary-material}

{This section contains some basic tools from Malliavin calculus and rough paths theory, as well as some analytical results, which are crucial for the definition and analysis of equation \eqref{eq:rde-intro}.}

\subsection{Preliminaries on rough paths}\label{sec:rough-path-above-X}

In this section we shall recall the notion of a rough path and how this applies to Gaussian signals. The interested reader is referred to \cite{FH,FV-bk,Gu} for further details.

For $s<t$ and $m\geq 1$, consider the simplex $\Delta_{st}^{m}=\{(u_{1},\ldots,u_{m})\in\lbrack s,t]^{m};\,u_{1}<\cdots<u_{m}\} $, while a simplex over $[0,T]$ will be denoted by $\Delta^{m}$. For a generic finite dimensional vector space  $E$, for an $E$-valued function $f$ defined on $[0,T]$ and for all $(s,t)\in\Delta^{2}$ we set
\begin{equation*}
\delta f_{st} = f_{t} -f_{s} .
\end{equation*}
The notion of a rough path relies on the notion of the signature of a path, that we now proceed to recall.

%%%%%%%%%%%%%%%%%%%%%%%%%%%%%%%%%%%%%%%%%%%%%%%%%%%

We start by defining an algebra in which the signature of a rough path will live.

\begin{definition}\label{def:truncated-algebra}
For $N\in\mathbb{N}$, the truncated algebra $T^{N}(\mathbb{R}%
^{d})$ is defined by 
\begin{equation*}
T^{N}(\mathbb{R}^{d})=\bigoplus_{n=0}^{N}(\mathbb{R}^{d})^{\otimes n},
\end{equation*}
with the convention $(\mathbb{R}^{d})^{\otimes
0}=\mathbb{R}$. The set $T^{N}(\mathbb{R}^{d})$ is equipped with a straightforward
vector space structure, plus an operation $\otimes$ defined by
\[
\lc g\otimes h\rc^{n}=\sum_{k=0}^{N}g^{n-k}\otimes h^{k},\qquad g,h\in
T^{N}(\mathbb{R}^{d}),
\]
where $g^{n}$ designates the projection on the $n$-th tensor level. 
\end{definition}

\noindent
Notice that with Definition \ref{def:truncated-algebra} in hand,
$(T^{N}(\mathbb{R}^{d}),+,\otimes)$ is an associative algebra with unit
element $\mathbf{1} \in (\mathbb{R}^{d})^{\otimes 0}$.

In the sequel we consider the process $X$ driving equation \eqref{eq:rde-intro} as a special case of continuous $\mathbb{R}^{d}$-valued paths defined on $[0,T]$. The regularity of $X$ will often be characterized by its $p$-variation.

\begin{definition}\label{def:var-norms-on-C2}
Let $f$ be a continuous function from $[0,T]$ to a finite dimensional vector space $E$. For all $p>0$ we set
$$
\|f\|_{p-{\rm var}}
=
\|f\|_{p-{\rm var}; [0,T]}
=
\sup_{\Pi \subset [0,T]}\left(\sum_{i} |f_{t_{i}t_{i+1}}|^p\right)^{1/p},
$$
where the supremum is taken over all subdivisions $\Pi$ of $[0,T]$. The set of continuous functions with finite $p$-variation is denoted by $\cac^{p{\rm -var}}(E)$.
\end{definition}

\noindent
Related to finite $p$-variation functions we will also consider the set of $\gamma$-H\"{o}lder
continuous functions,  and we recall that this functional space is defined by \eqref{eq:def-holder-norms}.

With our simplex notation $\Delta$ and Definition \ref{def:truncated-algebra} in mind, a continuous map
$\mathbf{x}:\Delta^{2}\rightarrow T^{N}(\mathbb{R}^{d})$ is called a
multiplicative functional if for $s<u<t$ one has $\mathbf{x}_{s,t}%
=\mathbf{x}_{s,u}\otimes\mathbf{x}_{u,t}$. A particular occurrence of this kind of map is given when one considers a path $x$ with finite
variation and sets for $(s,t)\in\Delta^{2}$,
\begin{equation}
\mathbf{x}_{st}^{n}=\sum_{1\leq i_{1},\ldots,i_{n}\leq d}\biggl( \int%
_{\Delta_{st}^{n}}dx^{i_{1}}\cdots dx^{i_{n}}\biggr) \,e_{i_{1}}\otimes
\cdots\otimes e_{i_{n}},\label{eq:def-iterated-intg}%
\end{equation}
where $\{e_{1},\ldots,e_{d}\}$ denotes the canonical basis of $\mathbb{R}^{d}
$. Then the so-called \textit{signature} of $x$ is the following object:
\begin{equation}\label{eq:signature-smooth-x}
S_{N}(x):\Delta^{2}\rightarrow T^{N}(\mathbb{R}^{d}),\qquad(s,t)\mapsto
S_{N}(x)_{s,t}:=1+\sum_{n=1}^{N}\mathbf{x}_{st}^{n}.
\end{equation}
It is worth mentioning that $S_{N}(x)$ will be our typical example of multiplicative functional. In addition, it can be shown that $S_{N}(x)$ lives in a subset $G^{N}(\mathbb{R}^{d})$ of $T^{N}(\mathbb{R}^{d})$ consisting of \emph{group-like elements}. This subset is defined by
\begin{equation*}
G^{N}(\mathbb{R}^{d}) = \exp^\oplus\bigl(L^N(\mathbb{R}^d)\bigr),
\end{equation*}
where $L^N(\mathbb{R}^d)$ is the linear span of all elements
that can be written as a commutator of the type $a \otimes b - b\otimes a$ for two elements in $T^{N}(\mathbb{R}^{d})$. It is known that there is a Carnot-Caratheodory norm  on $G^N(\mr^d)$ (see \cite{FV-bk}), which  we denote by $\|\cdot\|_{CC}$. It is a homogeneous norm with respect to the natural scaling operation on $T^N(\mr^d)$. 

A rough path can be seen as a generalization of the signature \eqref{eq:signature-smooth-x} to non smooth situations. Specifically, the definition of rough path can be summarized as follows.
 \begin{definition}\label{def:RP}
The space of weakly geometric $p$-rough paths 
is the set of multiplicative paths $\mathbf{x}:\Delta^{2}\rightarrow G^{\lfloor
p\rfloor}( \mathbb{R} ^{d}) $ such that the following norm is finite:
\begin{equation}\label{inhomogeneous}
\|\mathbf{x}\|_{p-\mathrm{var};[0,T]}=\sup_{\Pi \subset [0,T]}\left(\sum_{i} \|\mathbf{x}_{t_{i}t_{i+1}}\|_{CC}^p\right)^{1/p},
\end{equation}
where the supremum is taken over all subdivisions $\Pi$ of $[0,T]$.
An important subclass of weakly geometric $p$-rough paths is the set of geometric $p$-rough paths. These are multiplicative paths $\mathbf{x}:\Delta^{2}\rightarrow G^{\lfloor
p\rfloor}( \mathbb{R} ^{d}) $ with $\|\mathbf{x}\|_{p-\mathrm{var};[0,T]}<\infty$ such that there exists a sequence 
$\{  x^{\ep}; \, \ep>0 \} $ with $x^{\ep}\in C^{\infty}([0,T];\mathbb{R}^{d})$ and $S_{\lfloor p\rfloor}(x^\epsilon)\to\mathbf{x}$ under $\|\cdot\|_{p-\mathrm{var;[0,T]}}$. %satisfying 
%\begin{equation}\label{eq:closure-smooth-path}
%\lim_{\ep\to 0} \|\mathbf{x}^{-1}S_{N}(x_{\ep})\| =0.
%\end{equation}
In other words, the set of geometric $p$-rough paths is the closure of smooth paths in the space of weakly geometric $p$-rough paths equipped with a $p$-var norm.

%{\color{blue}Maybe, it is better to introduce the CC distance (when we introduced the group $G^N(\mr^d)$), and define the norm in (6) using the CC distance,
%$$\|\mathbf{x}\|_{p-\mathrm{var};[0,T]}=\sup_{\Pi \subset [0,T]}\left(\sum_{i} \|\mathbf{x}_{t_{i}t_{i+1}}\|_{CC}^p\right)^{1/p}.$$
%This way, the p-var norm for $\mathbf{x}$ is more consistent with Definition 2.2, and there is no need to introduce the new notation $\mathcal{N}_{p; [0,T]}$}

%\hre{I'm OK with that. I thought \eqref{inhomogeneous} was shorter to introduce but we can go with CC as well.}
\end{definition}

%%%%%%%%%%%%%%%%%%%%%%%%%%%%%%%%%%%%%%%%%%%%%%%%%%%

%We can now state the basic assumption on the kind of path $x$ we will consider in the sequel.

%\begin{hypothesis}\label{hyp:nth-iterated-intg-x}
%Let $x$ be a continuous $\R^d$-valued  path with finite $p$-variation for $p\ge 1$. We assume that $x$  gives rise to a geometric rough path in the sense of Definition \ref{def:RP}.
%\end{hypothesis}

With the above preliminary notions in hand, we now give the main theorem concerning existence and uniqueness of the solution to a rough differential equation. We refer the reader to \cite{FH,Gu} for its proof.

\begin{theorem}\label{thm:exist-uniq-rde-rough}
Suppose $x$ can be lifted as a geometric $p$-rough path % process satisfying Hypothesis \ref{hyp:nth-iterated-intg-x} 
and $V_0,\ldots,V_d$ be $\cac^\gamma$-Lipschitz continuous vector fields in $\R^{n}$ for some $\gamma> p \geq1$. For $\ep>0$, let $z^{\ep}$ be the unique solution of the following ordinary differential equation on $[0,T]$
\begin{equation}\label{eq:rde-approx} 
z^{\ep}_t =z +\int_0^t V_0 (z^{\ep}_s)ds+
\sum_{i=1}^d \int_0^t V_i (z^{\ep}_s) dx^{\ep,i}_s,
\end{equation}
where $x^{\ep}$ is a sequence of smooth functions approximating $x$ in the sense of Definition \ref{def:RP}. Then $z^{\ep}$ converges in $p$-variation to a path $z$, which can be seen as the unique solution of equation~\eqref{eq:rde-intro} understood in rough path sense.
 \end{theorem}

In the remainder of the article we assume that $X_t=(X_t^1,...,X_t^d)$ is a continuous, centered Gaussian process with i.i.d. components, defined on a complete probability space $(\Omega, \cf, \mp)$. The covariance function of $X$ is defined as follows
\begin{equation}\label{eq:def-covariance-X}
R(s,t):=\EE\lc X_{s}^{j} X_{t}^{j}\rc, 
\end{equation}
where $X^{j}$ is any of the components of $X$. We shall also use the following notation in the sequel
\begin{equation}\label{eq:def-variance-Xt}
\si_{t}^{2} := \EE\lc  \lp  X_{t}^{j} \rp^{2} \rc,
\quad\text{and}\quad
\si_{s,t}^{2} := \EE\lc  \lp  \delta X_{st}^{j} \rp^{2} \rc .
\end{equation}
A lot of the information concerning $X$ is encoded in the rectangular increments of the covariance function $R$, which is given by
\begin{equation}\label{eq:rect-increment-cov-fct}
R_{uv}^{st} := \EE\lc (X_t^j-X_s^j) \, (X_v^j-X_u^j) \rc.
\end{equation}
The 2D $\rho$-variation of $R$ on a rectangle $[0,t]^2$ is given by
\begin{equation}\label{eq:def-2d-rho-var}
V_{\rho}(R; [0,t]^2) := 
\sup \lcl  
\lp \sum_{i,j} \lln R_{s_{i} s_{i+1}}^{t_{j}t_{j+1}} \rrn^{\rho} \rp^{1/\rho}; \, (s_i), (t_j)\in \Pi
\rcl,
\end{equation}
where $\Pi$ is the set of partitions of {\color{blue}$[0,t]$}. For simplicity, we denote $V_\rho(R)=V_\rho(R;[0,T]^2)$ in the following. %Our basic assumption on $R$, which will prevail until the end of the article, is the following:
%\begin{hypothesis}\label{hyp:rho-var-R}
%The function $R$ defined by \eqref{eq:def-covariance-X} admits a finite $2$-dimensional $\rho$-variation for a given $\rho\in[0,2)$.
%\end{hypothesis}  
The following result (borrowed from \cite{FV-bk}) relates the $\rho$-variation of $R$ with the pathwise assumptions allowing to apply the abstract rough paths theory.

\begin{proposition}\label{prop:Gaussian-rough-path}
Let $X=(X^1,\ldots,X^d)$ be a continuous, centered Gaussian process with  i.i.d.\ components and covariance function $R$ defined by \eqref{eq:def-covariance-X}. If $R$ has finite 2D $\rho$-variation for some $\rho\in[1,2)$, {then almost surely $X$ can be lifted to a geometric $p$-rough path with $p>2\rho$.} 
\end{proposition}

\noindent
As a direct application of Theorem \ref{thm:exist-uniq-rde-rough} and Proposition \ref{prop:Gaussian-rough-path}, we notice that whenever a Gaussian process $X$ admits a covariance function $R$ with finite $2D$ $\rho$-variation (and $\rho\in[1,2)$), then, almost surely, equation \eqref{eq:rde-intro} driven by $X$ admits a unique solution in the rough path sense. In the sequel we shall give some information about the law of this solution $Z$.

\subsection{Wiener space associated to Gaussian processes}\label{sec:wiener-space-general}

In this section we consider again the continuous, centered Gaussian process $X$ of Section \ref{sec:rough-path-above-X}. Recall that its covariance function $R$ is defined by \eqref{eq:def-covariance-X}. Our analysis is based on two different (though related) Hilbert spaces $\tch$ and $\ch$. Roughly speaking, the space $\ch$ is the usual Cameron-Martin  space of $X$, while $\tch$ is the space allowing a proper definition of Wiener integrals as defined e.g in \cite{Nu06}.

The Cameron-Martin space $\ch$ is defined to be the
completion of the linear space of functions of the form
\[
\lcl
\sum_{i=1}^{n}a_{i}R\left(  t_{i},\cdot\right)  ,\quad a_{i}\in
\mathbb{R}
\text{ and }t_{i}\in\left[  0,T\right]  \rcl ,
\]
with respect to the following inner product
\begin{equation}\label{eq:def-inner-pdt-bar-H}
\left\langle \sum_{i=1}^{n}a_{i}R\left(  t_{i},\cdot\right)  ,\sum_{j=1}%
^{m}b_{j}R\left(  s_{j},\cdot\right)  \right\rangle _{\ch}=\sum
_{i=1}^{n}\sum_{j=1}^{m}a_{i}b_{j}R\left(  t_{i},s_{j}\right)  .
\end{equation}
The space $\tch$ is defined similarly, but this time we are considering the completion of  the set of step functions
\[
\mathcal{E}
=
\left\{  \sum_{i=1}^{n}a_{i} \1_{\left[  0,t_{i}\right]  }:a_{i}\in%
\mathbb{R}
\text{, }t_{i}\in\left[  0,T\right]  \right\}  ,
\]
with respect to  the inner product%
\begin{equation}\label{eq:def-inner-pdt-H}
\left\langle \sum_{i=1}^{n} a_{i} \1_{[0,t_{i}]}  ,
\sum_{j=1}^{m}b_{j} \1_{[0,s_{j}]}  \right\rangle _{\tch}
=
\sum_{i=1}^{n}\sum_{j=1}^{m}a_{i}b_{j}R\left(  t_{i},s_{j}\right) .
\end{equation}
{
\begin{remark}\label{representation H norm}
Let $X_0=0$ and thus $R(0,0)=0$. Then, as suggested by \eqref{eq:def-inner-pdt-H}, for any $h_1, h_2\in\tch$, we have
\begin{align}\label{rep H norm}\langle h_1, h_2\rangle_\tch=\int_0^T\int_0^Th_1(s)h_2(t)dR(s,t),\end{align}
whenever the 2D Young's integral on the right-hand side is well-defined (see, e.g., \cite[Proposition 4]{CFV} for details).
\end{remark}
}

Since $\tch$ is the completion of $\ce$ w.r.t $\left\langle
\cdot,\cdot\right\rangle _{\tch}$, it is obvious that the
linear map $\crr:\mathcal{E}\rightarrow\ch$ defined by
\begin{equation}
\crr\left(  \1_{\left[  0,t\right]  }\right)  =R\left(  t,\cdot\right)
\label{isom}%
\end{equation}
extends to an isometry between $\tch$ and $\ch$. 

Let $H^{1}\left(  X\right)  \subseteq L^{2}\left(  \Omega,\mathcal{F},\mp\right)  $ be the $\left\vert \cdot\right\vert _{L^{2}\left(
\Omega\right)  }$-closure of the set%
\[
\left\{  \sum\nolimits_{i=1}^{n}a_{i}X_{t_{i}}^1:a_{i}\in%
%TCIMACRO{\U{211d} }%
%BeginExpansion
\mathbb{R}
%EndExpansion
,\text{ }t_{i}\in\left[  0,T\right]  ,\text{ }n\in%
%TCIMACRO{\U{2115} }%
%BeginExpansion
\mathbb{N}
%EndExpansion
\right\}  .
\]
We also recall that $\tch$ is isometric to $H^1(X)$ through the Wiener integral $X(\phi), \phi\in\tch$, where, in particular, we have that $X(\1_{[0,t]})=X_t$.

\begin{remark}\label{rmk:H-on-subinterval}
Since the space $\tch$ is a closure of indicator functions, it is easily defined on any interval $[a,b]\subset[0,T]$. We denote by $\tch([a,b])$ this restriction. For $[a,b]\subset[0,T]$, one can then check the following identity by a limiting procedure on simple functions
\begin{equation}\label{eq:norm-H-as-2d-young}
\lla f \, \1_{[a,b]} , \, g \, \1_{[a,b]} \rra_{\tch}
=
\lla f  , \, g  \rra_{\tch([a,b])}.
\end{equation}
\end{remark}

The rough path analysis of Gaussian processes relies heavily on embedding results for the Cameron-Martin space $\ch$ into spaces of functions of finite $p$-variation. In the following we shall recall a recent embedding result from \cite{FGGR13}. To this aim, %let us first recall that for any function $h:[0,T]\to\R$ we define $\delta h_{st}:=h_{t}-h_{s}$ for all $(s,t)\in\Delta_{2}$. 
 let us recall the definition of the mixed $(\gamma,\rho)$-variation
given in \cite{Tow02}.

\begin{definition}\label{def:mixed-variation}
For a general continuous function $R:[0,T]^{2}\to\R$ and two parameters $\gamma,\rho\geq1$, we set 
\begin{align} 
V_{\gamma,\rho}(R;[s,t]\times[u,v]):
=\sup_{\substack{(t_{i})\in\mathcal{D}([s,t])\\
(t_{j}^{\prime})\in\mathcal{D}\left(\left[u,v\right]\right)
}
}\left(\sum_{t'_{j}}\left(\sum_{t_{i}}\left|R_{t_{i}t_{i+1}}^{t_{j}^{\prime}t_{j+1}^{\prime}}\right|^{\gamma}\right)^{\frac{\rho}{\gamma}}\right)^{\frac{1}{\rho}},\label{eq:mixed_var}
\end{align}
where $\mathcal{D}([s,t])$ denotes the set of all dissections of $[s,t]$ and where we have set
\begin{equation*}
R_{t_{i}t_{i+1}}^{t_{j}^{\prime}t_{j+1}^{\prime}}
=
R(t_{i+1},t_{j+1}^{\prime}) - R(t_{i+1},t_{j}^{\prime}) - R(t_{i},t_{j+1}^{\prime}) + R(t_{i},t_{j}^{\prime}).
\end{equation*}
\end{definition}

\noindent
Observe that, whenever the function $R$ in Definition \ref{def:mixed-variation} is given as a covariance function as in \eqref{eq:def-covariance-X}, then the rectangular increment $R_{t_{i}t_{i+1}}^{t_{j}^{\prime}t_{j+1}^{\prime}}$ is given by \eqref{eq:rect-increment-cov-fct}. In addition, the $\rho$-variation of $R$ introduced in \eqref{eq:def-2d-rho-var} and invoked in Proposition \ref{prop:Gaussian-rough-path} is recovered as $V_{\rho}=V_{\rho,\rho}
$. 
As a last elementary remark, also notice that  
  $$V_{\gamma\vee\rho}(R;A)\leq V_{\gamma,\rho}(R;A)\leq V_{\gamma\wedge\rho}(R;A),
  $$ 
for all rectangles $A\subseteq[0,T]^{2}$. We set, for future use
\begin{equation}\label{eq:def-kappa}
\kappa_{s,t}^{2}:= V_{1,\rho}(R;[s,t]^{2}),
\quad\text{and}\quad
\kappa_{t}^{2}:= V_{1,\rho}(R;[0,t]^{2}).
\end{equation}

With these elementary notions at hand, we next introduce an hypothesis which allows the use of both rough paths techniques and tools from stochastic analysis for the underlying process. 

\begin{hypothesis}\label{hyp:mixed-var-R}
Let $X$ be a $d$-dimensional continuous, centered Gaussian process with i.i.d.\ components and covariance $R$ defined by \eqref{eq:def-covariance-X}. We assume that the function $R$ admits a finite mixed $(1,\rho)$-variation, as introduced in Definition \ref{def:mixed-variation}, for some $\rho\in[1,2)$.
\end{hypothesis}

\begin{remark}\label{rmk: existence rp}
Since the mixed $(1,\rho)$-variation of $R$ controls $V_\rho(R)$, Proposition \ref{prop:Gaussian-rough-path} and Hypothesis \ref{hyp:mixed-var-R} imply the existence of a rough path lift of $X$ to a geometric $p$-rough path with $p>2\rho$.
\end{remark} 

\begin{definition}\label{def:holder-controlled}
Given $\rho\in[1,2)$, we say that $R$ has finite H\"{o}lder-controlled mixed $(1,\rho)$-variation if there exists a $C>0$ such that for all $0\leq s\leq t\leq T$ we have 
$$V_{1,\rho}(R; [s,t]^2)\leq C(t-s)^{1/\rho}.$$
\end{definition}

\begin{remark}
An important consequence of $R$ having finite H\"{o}lder-controlled mixed $(1,\rho)$-variation is that $\bf{X}$ has $1/p$-H\"{o}lder continuous sample paths for every $p>2\rho$. {This will be needed in order to obtain the interpolation inequality in Proposition \ref{th: interpolation} below which plays an important role in the analysis.}
\end{remark}

\begin{remark}
Similarly to the argument in \cite[Remark 2.4]{CHLT}, for any process $X$ satisfying Hypothesis \ref{hyp:mixed-var-R}, one can introduce a deterministic time-change $\tau:[0,T]\to[0,T]$ such that $\tilde{X}=X\circ\tau$ has finite H\"{o}lder-controlled mixed $(1,\rho)$-variation.
\end{remark}
  
We are now ready to recall an embedding result for the Cameron-Martin space $\ch$, obtained in \cite{FGGR13}.

\begin{theorem}\label{thm:refined_CM_embedding} 
Let $X$ be a centered Gaussian process satisfying Hypothesis \ref{hyp:mixed-var-R} and recall that $\ch$ is defined by the inner product \eqref{eq:def-inner-pdt-bar-H}. Then there is a
continuous embedding 
\begin{equation*}
\ch\hookrightarrow C^{q\mathrm{-var}},
\quad\text{ with }\quad
q=\frac{1}{\frac{1}{2\rho}+\frac{1}{2}}<2. 
\end{equation*}
More precisely, the following inequality holds true
\begin{align*}
\|h\|_{q-\mathrm{var};[s,t]}\leq \kappa_{s,t} \, \|h\|_{\ch},
\quad\forall[s,t]\subseteq[0,T],
\end{align*}
where the constant $\kappa_{s,t}$ is defined by \eqref{eq:def-kappa}.
\end{theorem}

Finally we can give a statement which will be the basis of the interpretation of several integrals related to Malliavin derivatives %\hb{(and hopefully answers the questions about Young complementarity)}.

\begin{corollary}
Let $X$ be a centered Gaussian process satisfying Hypothesis \ref{hyp:mixed-var-R} for a given $\rho\in[1,2)$, let $\ch$ be the Cameron-Martin space related to $X$ and let $\ep\in(0,2-\rho]$ small enough. Then

\noindent
\emph{(i)}
The process $X$ gives rise to a geometric $p$-rough path for $p=2\rho+\ep$. 

\noindent
\emph{(ii)}
The spaces $\ch$ and $\cac^{p-{\rm var}}$ satisfy Young's complementary condition: There exists a $q$ such that $\ch$ is embedded in $\cac^{q-{\rm var}}$ and such that $p^{-1}+q^{-1}>1$.
\end{corollary}

\begin{proof}
Item (i) follows from Remark \ref{rmk: existence rp}. As far as item (ii) is concerned, we invoke Theorem \ref{thm:refined_CM_embedding} and we take $q=(\frac{1}{2\rho}+\frac{1}{2})^{-1}$. Since $\rho<2$ and since we have chosen $p=2\rho+\ep$ with $\ep$ small enough, it is easily checked that $p^{-1}+q^{-1}>1$.

\end{proof}

{
\begin{remark}\label{rmk:tch-identities}
   Let $X$ be a Gaussian process starting at zero and satisfying Hypothesis \ref{hyp:mixed-var-R} for a given $\rho\in[1,2)$. 
   \begin{enumerate}
   \item Let $f\in C^{p-var}([0,T])$ with $\frac{1}{p}+\frac{1}{\rho}>1$. Then, $f\in \tch$ and
     \begin{equation}\label{eq:ch-eq}
       \|f\|^2_\tch = \int_0^T\int_0^Tf(s)f(t)dR(s,t),
     \end{equation}
   where the right hand side is well-defined as a 2D-Young integral.
   In particular, $C^{p-var}([0,T]) \hookrightarrow \tch$ with 
     $$\|f\|^2_\tch \le C \|f\|_{p-var}^2 \|R\|_{\rho-var}.$$
   \item Let $h_1 \in C^{p-var}([0,T])$ with $\frac{1}{p}+\frac{1}{\rho}>1$, and $h_2\in\tch$. Then,
        $$\langle h_1, h_2 \rangle_{\tch([0,t])} =\int_0^t h_1 d \crr h_2,$$
      where the right hand side is well-defined as a Young-integral and where we recall that the isomorphism $\crr$ is defined by \eqref{isom}. 
   \end{enumerate}
\end{remark}
\begin{proof}
 (1): Since $R$ has finite 2D-$\rho$ variation and $f$ has finite $p$-variation with $\frac{1}{p}+\frac{1}{\rho}>1$ the 2D-Young integral is well-defined. For elementary processes, \eqref{eq:ch-eq} is immediate. Then, approximating $f$ by piece-wise constant approximants $f^n$, and noting that (cf.\ \cite[Exercise 6.9]{FV}) 
   $$\int_0^T\int_0^Tf^n(s)f^n(t)dR(s,t)\to \int_0^T\int_0^Tf(s)f(t)dR(s,t)$$ 
 the sequence $f^n$ is a Cauchy sequence (identified with $f$) in $\tch$ and the identity carries over to the limit.
 
 \noindent
 (2): By (1), $h_1 \in \tch$ and  $\crr h_2 \in \ch$ and, by Theorem \ref{thm:refined_CM_embedding}, $\crr h_2 \in C^{q-var}$ with $q=\frac{1}{\frac{1}{2\rho}+\frac{1}{2}}$. Since $\frac{1}{p}+\frac{1}{q} >1$, using $\rho\ge1$, this implies that the Young integral is well-defined. Since the identify is readily checked for elementary functions, the claim follows as in (1). 
\end{proof}
}

\subsection{Interpolation inequalities}

Interpolation inequalities involving Cameron-Martin spa\-ces are crucial in order to bound Malliavin derivatives which appear in density formulae. In this section we derive such inequalities for a general Gaussian process, under conditions introduced in \cite{CHLT,FGGR13}. The first condition we shall impose concerns correlations of increments.

\begin{hypothesis}\label{hyp:correlation-increments-X}
Let $X$ be an $\R^{d}$-valued centered Gaussian process $X$ with i.i.d.\ coordinates and covariance function $R$. In the sequel we assume that:

\noindent\emph{(i)}
$X$ has non-positively correlated increments, that is, for all $(t_{1},t_2,t_3,t_{4})\in\Delta^{4}$ and every coordinate $j=1,\ldots,d$ we have
\begin{equation}\label{eq:hyp-neg-correlated-increments}
R_{t_{1}t_{2}}^{t_{3}t_{4}}
=
\EE\lc \delta X_{t_{1}t_{2}}^{j}  \, \delta X_{t_{3}t_{4}}^{j} \rc \le 0.
\end{equation}

\noindent\emph{(ii)}
The covariance $R$ is diagonally dominant. That is, for all $(t_{1},t_2,t_3,t_{4})\in\Delta^{4}$ and every coordinate $j=1,\ldots,d$ we have
\begin{equation}\label{eq:hyp-diag-dominance-R}
R_{t_{2}t_{3}}^{t_{1}t_{4}}
=
\EE\lc \delta X_{t_{2}t_{3}}^{j}  \, \delta X_{t_{1}t_{4}}^{j} \rc \ge 0.
\end{equation}
\end{hypothesis}

With this Hypothesis at hand, we start with some inequalities which stem from the Cameron-Martin embedding Theorem~\ref{thm:refined_CM_embedding}.

\begin{proposition}\label{interpolation-general} 
Let $X$ be a Gaussian process starting from zero and satisfying Hypothesis \ref{hyp:mixed-var-R}. Further, {let %$q=(\frac{1}{2\rho}+\frac{1}{2})^{-1}$ and consider 
$p\ge 1$ such that $1/p+1/\rho>1$.} Then 

\noindent\emph{(i)}
There exist  constants $c_1,c_2 >0$ such that for every {$f \in C^{p-var}([0,T])$}  and $t \in (0,T]$, we have
\[
 \| f \mathbf{1}_{[0,t]}  \|^2_\tch\le c_2  \, 
 \ka_{t}^{2} \lp \| f \mathbf{1}_{[0,t]}\|_{p-var}^2+\|f \mathbf{1}_{[0,t]}\|_\infty^2 \rp ,
\]
where $\ka_{t}$ is as in \eqref{eq:def-kappa}.% \hb{We have already argued that (ii) below is not used in the remainder of the paper. But are we using (i) somewhere?}

\noindent\emph{(ii)}
Assume that $X$ satisfies Hypothesis \ref{hyp:correlation-increments-X} and let $\mathcal{C}^\gamma$ be the space of $\gamma$-H\"{o}lder continuous functions. Then, for any {$f\in \mathcal{C}^\gamma$} with $1/\rho+\gamma>1$,
\begin{equation}\label{eq:low-bnd-norm-H-general}
  \|f \mathbf{1}_{[0,t]}\|^2_\tch 
  \ge \int_0^t f^2(r) R(dr,t) 
  \ge \si_{t}^{2} \min_{[0,t]} |f|^2  ,
\end{equation}
where $\si_{t}^{2}$ is as in \eqref{eq:def-variance-Xt}.
\end{proposition}

\begin{remark}
Equation \eqref{eq:low-bnd-norm-H-general} above is in fact a consequence of \cite[Proposition 6.6]{CHLT}, by taking $s=0$ and $t=T$ therein. We have included a more elementary proof here for sake of clarity.
\end{remark}

\begin{proof}[Proof of Proposition \ref{interpolation-general}]

We prove the two items of this proposition separately.

\noindent\emph{Proof of (i).}
Recall that the spaces $\tch([a,b])$ are introduced in Remark \ref{rmk:H-on-subinterval}. { By Remark \ref{rmk:tch-identities}}, the following relation holds true, for any $h_1,h_2 \in C^{p-var}([0,T])$, 
\[
\langle h_1, h_2 \rangle_{\tch([0,t])} =\int_0^t h_1 d \crr h_2.
\]
Hence, if $p^{-1}+q^{-1}>1$, classical inequalities for Young's integral imply
\begin{equation}\label{a1}
| \langle h_1, h_2 \rangle_{\tch([0,t])} | 
\le 
C( \| h_1\|_{p-var;[0,t]} +\| h_1\|_{\infty;[0,t]} ) \| \crr h_2 \|_{q-var;[0,t]}.
\end{equation}
We now use Theorem \ref{thm:refined_CM_embedding} to get the bound
\begin{equation*}
 \| \crr h_2 \|_{q-var;[0,t]} 
  \le \ka_{t} \, \|\crr h_2\|_{\ch([0,t])}  
 = \ka_{t} \, \|h_2\|_{\tch([0,t])},
\end{equation*}
where we recall that we have set $\ka_{t}^{2}=V_{1,\rho}(R;[0,t]^{2})$.
Plugging this information back into~\eqref{a1} and choosing $h_{1}=h_{2}=f$, we obtain
\begin{align*}   
	\| f   \|_{\tch([0,t])}^2 
	&= | \langle f, f \rangle_{\tch([0,t])} | 
	\le C( \| f\|_{p-var;[0,t]} +\| f\|_{\infty;[0,t]} ) \| \crr f \|_{q-var;[0,t]} \\
	&\le C \ka_{t} ( \| f\|_{p-var;[0,t]} +\| f\|_{\infty;[0,t]} )  \|f\|_{\tch([0,t])}.
\end{align*}
Dividing this expression by $\|f\|_{\tch([0,t])}$ finishes the proof of claim (i).

\noindent\emph{Proof of (ii).} 
We first prove the claim for elementary step functions. Namely, consider $t\le T$, a partition $(t_i)$ of the interval  $[0,t]$, and set 
\begin{equation*}
f\mathbf{1}_{[0,t]}=\sum_i a_i  \mathbf{1}_{[t_i,t_{i+1}]} .
\end{equation*}
Then the following identity obviously holds true
\begin{equation*}
 \|f\mathbf{1}_{[0,t]}\|_\tch^2
 =
 \sum_{i,j}a_i a_j \lla \mathbf{1}_{[t_i,t_{i+1}]},\mathbf{1}_{[t_j,t_{j+1}]}\rra_\tch
 =
 \sum_{i,j}a_i a_j R_{t_{j}, t_{j+1}}^{t_{i}, t_{i+1}}.
\end{equation*}
We now separate diagonal and non-diagonal terms in order to get
\begin{equation}\label{a2}
\|f\mathbf{1}_{[0,t]}\|_\tch^2 
= 
\sum_{i}\sum_{j\ne i}a_i a_j R_{t_{j}, t_{j+1}}^{t_{i}, t_{i+1}} 
	 + \sum_{i}a_i^2 R_{t_{i}, t_{i+1}}^{t_{i}, t_{i+1}}
\ge S_{1}- S_{2},
\end{equation}
where $S_{1}$ and $S_{2}$ are defined by
\begin{equation*}
S_{1} = \sum_{i}a_i^2 R_{t_{i}, t_{i+1}}^{t_{i}, t_{i+1}},
\quad\text{and}\quad
S_{2} = \sum_{i}\sum_{j\ne i}|a_i| |a_j| \left|R_{t_{j}, t_{j+1}}^{t_{i}, t_{i+1}} \right|.
\end{equation*}
Next, in order to bound $S_{2}$ from above, we first invoke the elementary inequality $2 |a_i| |a_j| \le |a_i|^{2} + |a_j|^{2}$ to get
\begin{equation*}
S_{2} \le \frac{1}{2}\sum_{i}\sum_{j\ne i}a_i^2 \left|R_{t_{j}, t_{j+1}}^{t_{i}, t_{i+1}} \right|
		+\frac{1}{2}\sum_{i}\sum_{j\ne i} a_j^2 \left|R_{t_{j}, t_{j+1}}^{t_{i}, t_{i+1}} \right|.
\end{equation*}
Then, using \eqref{eq:hyp-neg-correlated-increments}, we get
\begin{equation*}
S_{2} \le
-\frac{1}{2}\sum_{i}\sum_{j\ne i}a_i^2 R_{t_{j}, t_{j+1}}^{t_{i}, t_{i+1}}
		-\frac{1}{2}\sum_{i}\sum_{j\ne i} a_j^2 R_{t_{j}, t_{j+1}}^{t_{i}, t_{i+1}}
		= -\sum_{i}\sum_{j\ne i}a_i^2 R_{t_{j}, t_{j+1}}^{t_{i}, t_{i+1}}.
\end{equation*}
Inserting this in \eqref{a2} yields
\begin{equation}\label{a3}
\|f\mathbf{1}_{[0,t]}\|_\tch^2 
   \ge \sum_{i,j}a_i^2  R_{t_{j}, t_{j+1}}^{t_{i}, t_{i+1}} 
   = \sum_{i}a_i^2  R_{0t}^{t_{i} t_{i+1}} .
\end{equation}
Let us observe that, owing to the diagonal dominance assumption \eqref{eq:hyp-diag-dominance-R}, the measure $R(dr,t)$ defined by
\begin{equation*}
R([u,v],t):= R_{0t}^{uv}
\end{equation*}
is non-negative. Furthermore, one can recast inequality \eqref{a3} as
\begin{equation*}
\|f\mathbf{1}_{[0,t]}\|_\tch^2 
   \ge
 \int_0^t f^2(r) R(dr,t) .
\end{equation*}
Using elementary properties of positive measures, we thus end up with
\begin{equation*}
   \|f\mathbf{1}_{[0,t]}\|_\tch^2 
   \ge  \min_{[0,t]} |f|^2 R_{0t}^{0t}
   =  \min_{[0,t]} |f|^2 \si_{t}^{2},
\end{equation*}
which proves the claim (ii) for elementary functions $f$.  {Finally, we show that the above remains true all $f\in\tch\cap \mathcal{C}^\gamma$.  Let $D=\{t_i: i=0,1,...,n\}$ be any partition of $[0,T]$, and set $f_D(t)=f(t_i), t_i\leq t<t_{i+1}$. Since $f_D$ is an elementary function, we have
$$\int_{[0,t]^2}f_D(s)f_D(t)dR(s,t)=\|f_D\mathbf{1}_{[0,t]}\|_\tch^2\geq\min_{[0,t]}|f_D|^2\sigma_t^2.$$
Note that we assume $f\in\mathcal{C}^\gamma$ with $1/\rho+\gamma>1$. The left hand-side of the above display is the Riemann sum approximation to the 2D Young integral of $f$ against $R$ along the partition $D$. Hence, if we shrink the mesh of the partition $D$, 
$$\int_{[0,t]^2}f_D(s)f_D(t)dR(s,t)\to\int_{[0,t]^2}f(s)f(t)dR(s,t)=\|f\mathbf{1}_{[0,t]}\|_\tch^2.$$
On the other hand, $\min_{[0,t]}|f_D|\to\min_{[0,t]}|f|$, when shrinking the mesh of $D$, by the construction of $f_D$ and the fact that $f$ is continuous.  The proof is thus completed.
}
\end{proof}

We now wish to get a non-degeneracy result for the norm in $\tch$, that is, a lower bound on $\|f\|_{\tch}$ involving $\|f\|_{\infty}$. {From \cite[Condition 2]{CHLT} we recall the following non-degeneracy condition.}

\begin{hypothesis} \label{assumption1}
Let $(X_t)_{t\in[0,T]}$ be a centered continuous $\R^{d}$-valued Gaussian process. For any $0\leq a\leq b\leq T$, denote by $\cf_{a,b}$ the following $\si$-algebra
\begin{equation*}
\cf_{a,b}=\sigma(\delta X_{uv}: a\leq u\leq v\leq b).
\end{equation*}
Then we assume that there exists an $\alpha>0$ such that
\begin{align}\label{non-determ}
\inf_{0\leq s<t\leq T}\frac{1}{(t-s)^\alpha}\var{\delta X_{st}|\cf_{0,s}\vee\cf_{t,T}}=c_X>0.
\end{align}
We call the smallest $\alpha$ that satisfies the above condition the index of non-determinism.
\end{hypothesis}

{ We recall that in  \cite[Lemma 4.1 and Lemma 4.2]{CHLT} the condition above is verified for any $\alpha \in (0,1]$ for the case of $X$ being a fractional Brownian motion with Hurst index $H\in (0,\frac{1}{2})$. }

\begin{remark}
Note that since we are working with Gaussian processes, the above conditional variance
$\var{\delta X_{st}|\cf_{0,s}\vee\cf_{t,T}}$ is deterministic. Moreover, assuming Hypothesis \ref{assumption1} holds true and setting $s=0$ in \eqref{non-determ}, the law of total variance gives us
$$
\si_{t}^{2}=\var{X_t}\geq \var{\delta X_{0t}|\cf_{0,0}\vee\cf_{t,T}}\geq c_Xt^\alpha ,
$$
with $\si_{t}^{2}$ as in \eqref{eq:def-variance-Xt}.
%\txx{red}{Maybe we can find some examples  (fBm is one of them) such that $R(t,t)$ and $t^\alpha$ have the same order.}
\end{remark}

With Hypothesis  \ref{assumption1} at hand, we borrow the following interpolation inequality from \cite[Corollary 6.10]{CHLT}.

\begin{proposition}\label{th: interpolation}
Let $(X_t)_{t\in[0,T]}$ be a continuous Gaussian process starting from zero with covariance function $R:[0,T]^2\to\mr$. Suppose Hypothesis \ref{hyp:correlation-increments-X} and \ref{assumption1} are satisfied. Furthermore, we assume that $R$ has finite H\"{o}lder-controlled mixed $(1,\rho)$-variation for some $\rho\in[1,2)$ in the sense of Definition \ref{def:holder-controlled}.  Then there exists a universal constant $c$ such that for any $f\in C^\gamma([0,T],\mr)$ with $\gamma+1/\rho>1$, we have
\begin{equation}\label{eq:interpolation-1}
\|f\|_{\infty;[0,T]}
\leq
2 \max\left\{
\frac{\|f\|_{\tch}}{\si_{T}}  , \
\frac{1}{\sqrt{c_{X}}}\|f\|_{\tch}^{\frac{2\gamma}{2\gamma+\alpha}} 
\|f\|_{\gamma;[0,T]}^{\frac{\alpha}{2\gamma+\alpha}}
\right\},
\end{equation}
where $c_{X}$ is the constant appearing in equation \eqref{non-determ} and $\si_{t}$ is defined by \eqref{eq:def-variance-Xt}. 
\end{proposition}

\begin{remark}
In \cite{CHLT}, relation \eqref{eq:interpolation-1} is proved under the following additional hypothesis
\begin{equation}\label{eq:positive-cdt-covariance}
\mathrm{Cov} (X_{s,t} X_{u,v}|\cf_{0,s}\vee\cf_{t,S})\geq 0,
\end{equation}
for any $[u,v]\subset[s,t]\subset[0,S]\subset[0,T]$. However, we are working here under the standing assumptions \eqref{eq:hyp-neg-correlated-increments}, \eqref{eq:hyp-diag-dominance-R} in Hypothesis \ref{hyp:correlation-increments-X}, and it is shown in \cite[Corollary 6.8]{CHLT} that \eqref{eq:hyp-neg-correlated-increments} together with \eqref{eq:hyp-diag-dominance-R} implies \eqref{eq:positive-cdt-covariance}.
\end{remark}

\begin{remark}
Our interpolation inequality \eqref{eq:interpolation-1} also reads as
\begin{equation}\label{interpolation2}
\|f\|_{\tch}
\ge
\frac{\si_{T} \|f\|_{\infty;[0,T]}}{2}
\min\left\{1, \
\frac{2\left(\frac{c_{X}}{2}\right)^{\frac{2\gamma+\alpha}{4\gamma}}}{\si_{T}}
\frac{\|f\|_{\infty;[0,T]}^{\frac{\alpha}{2\gamma}}}{\|f\|_{\gamma;[0,T]}^{\frac{\alpha}{2\gamma}}}
\right\}.
\end{equation}
In fact we will use a slight generalization of \eqref{interpolation2} in the sequel. Namely, for all $t\le T$, Remark~\ref{rmk:H-on-subinterval} asserts that $\|f \1_{[0,t]}\|_{\tch}=\|f \|_{\tch([0,t])}$. We thus get the following interpolation inequality
\begin{equation}\label{interpolation3}
\| f \1_{[0,t]} \|_{\tch}
\ge
\frac{\si_{t} \|f\|_{\infty;[0,t]}}{2}
\min\left\{1, \
\frac{2\left(\frac{c_{X}}{2}\right)^{\frac{2\gamma+\alpha}{4\gamma}}}{\si_{t}}
\frac{\|f\|_{\infty;[0,t]}^{\frac{\alpha}{2\gamma}}}{\|f\|_{\gamma;[0,t]}^{\frac{\alpha}{2\gamma}}}
\right\}.
\end{equation}
\end{remark}

\subsection{Malliavin calculus for Gaussian processes}

In this section we review some basic aspects of Malliavin calculus. The reader is referred to \cite{Nu06} for further details.

As before $X_t=(X_t^1,...,X_t^d)$ is a continuous, centered Gaussian process with i.i.d.\ components, defined on a complete probability space $(\Omega, \cf, \mp)$. For sake of simplicity, we assume that $\cf$ is generated by $\{X_{t}; \, t\in[0,T]\}$. An $\mathcal{F}$-measurable real
valued random variable $F$ is said to be cylindrical if it can be
written, for some $m\ge 1$, as
\begin{equation*}
F=f\lp  X_{t_1},\ldots,X_{t_m}\rp,
\quad\mbox{for}\quad
0\le t_1<\cdots<t_m \le T,
\end{equation*}
where $f:\mathbb{R}^m \rightarrow \mathbb{R}$ is a $C_b^{\infty}$ function. The set of cylindrical random variables is denoted by~$\mathcal{S}$. 

\smallskip

The Malliavin derivative is defined as follows: for $F \in \mathcal{S}$, the derivative of $F$ in the direction $h\in\tch$ is given by
\[
\mathbf{D}_h F=\sum_{i=1}^{m}  \frac{\partial f}{\partial
x_i} \left( X_{t_1},\ldots,X_{t_m}  \right) \, h_{t_i}.
\]
More generally, we can introduce iterated derivatives. Namely, if $F \in
\mathcal{S}$, we set
\[
\mathbf{D}^k_{h_1,\ldots,h_k} F = \mathbf{D}_{h_1} \ldots\mathbf{D}_{h_k} F.
\]
For any $p \geq 1$, it can be checked that the operator $\mathbf{D}^k$ is closable from
$\mathcal{S}$ into $\mathbf{L}^p(\oom;\tch^{\otimes k})$. We denote by
$\mathbb{D}^{k,p}(\tch)$ the closure of the class of
cylindrical random variables with respect to the norm
\[
\left\| F\right\| _{k,p}=\left( \mathbb{E}\left[|F|^{p}\right]
+\sum_{j=1}^k \mathbb{E}\left[ \left\| \mathbf{D}^j F\right\|
_{\tch^{\otimes j}}^{p}\right] \right) ^{\frac{1}{p}},
\]
and we also set $\mathbb{D}^{\infty}(\tch)=\cap_{p \geq 1} \cap_{k\geq 1} \mathbb{D}^{k,p}(\tch)$. The divergence operator $\delta^{\diamond}$ is then defined to be the adjoint operator of $\mathbf{D}$.

\smallskip

Estimates of Malliavin derivatives are crucial in order to get information about densities of random variables, and Malliavin matrices as well as non-degenerate random variables will feature importantly in the sequel.
\begin{definition}\label{non-deg}
Let $F=(F^1,\ldots , F^n)$ be a random vector whose components are in $\mathbb{D}^\infty(\tch)$. Define the Malliavin matrix of $F$ by
\begin{equation} \label{malmat}
\gamma_F=(\langle \mathbf{D}F^i, \mathbf{D}F^j\rangle_{\tch})_{1\leq i,j\leq n}.
\end{equation}
Then $F$ is called  {\it non-degenerate} if $\gamma_F$ is invertible $a.s.$ and
$$(\det \gamma_F)^{-1}\in \cap_{p\geq1}L^p(\Omega).$$
\end{definition}
\noindent
It is a classical result that the law of a non-degenerate random vector $F=(F^1, \ldots , F^n)$ admits a smooth density with respect to the Lebesgue measure on $\mr^n$.

\subsection{Differential equations driven by Gaussian processes}
Recall that we consider the following kind of equation
\begin{equation}\label{eq:sde} 
Z^{z}_t =z +\int_0^t V_0 (Z^z_s)ds+
\sum_{i=1}^d \int_0^t V_i (Z^{z}_s) dX^i_s,
\end{equation}
where the vector fields $V_0,\ldots,V_d$ are $\cac_b^\infty$-vector fields on $\R^n$ and $X$ is a continuous, centered Gaussian process with i.i.d.\ components. Throughout this section, we assume that the covariance $R$ has finite 2D $\rho$-variation for some $\rho\in[1,2)$. Hence, as mentioned in Section \ref{sec:rough-path-above-X}, Proposition \ref{prop:Gaussian-rough-path}  implies the existence and uniqueness of  a solution to \eqref{eq:sde}.

Once equation (\ref{eq:sde}) is solved, the vector $Z_t^z$ is a typical example of  random variable which can be differentiated in the Malliavin sense. We shall express this Malliavin derivative in terms of the Jacobian $\bj$ of the equation, which is defined by the relation $\bj_{t}^{ij}=\partial_{z_j}Z_t^{z,i}$. Setting $DV_{j}$ for the Jacobian of $V_{j}$ as a function from $\R^{n}$ to $\R^{n}$, let us recall that $\bj$ is the unique solution to the linear equation
\begin{equation}\label{eq:jacobian}
\bj_{t} = \id_{n} + \int_0^t DV_0 (Z^z_s) \, \bj_{s} \, ds+
\sum_{j=1}^d \int_0^t DV_j (Z^{z}_s) \, \bj_{s} \, dX^j_s.
\end{equation}

The following integrability and differentiability results are summarized from \cite{CF,CLL, Inahama}.

\begin{proposition}\label{prop:deriv-sde}
Let $X$ be a continuous, centered $\R^{d}$-valued Gaussian process with i.i.d.\ components and covariance function $R$ having finite 2D $\rho$-variation for some $\rho\in[1,2)$.
Consider the solution $Z^z$ to (\ref{eq:sde}) and suppose that the vector fields $V_i$ are $\cac_b^\infty$. Then

\noindent\emph{(i)}
For any $\eta\ge 1$, there exists a finite constant $c_\eta$ such that the Jacobian $\bj$ defined by \eqref{eq:jacobian} satisfies
\begin{equation}
\EE\lc  \Vert \bj \Vert^{\eta}_{p-{\rm var}; [0,T]} \rc = c_\eta.
\end{equation}

\noindent\emph{(ii)}
For every $i=1,\ldots,n$, $t>0$, and $z \in \mathbb{R}^n$, we have $Z_t^{z,i} \in
\mathbb{D}^{\infty}(\tch)$ and the Malliavin derivative of $Z^z_t$ { can be realized as a function $\bd_sZ_t^z$ in $s\in[0,T]$ which satisfies 
\begin{equation}\label{eq:rep-malliavin-with-jacobian}
\mathbf{D}^j_s Z_t^{z}= \mathbf{J}_{s,t} V_j (Z^z_s),
\end{equation}
for all $j=1,\ldots,d$, $0\leq s \leq t$ and
$$\bd_s^jZ^z_t=0,$$
for all $s>t$.} Here $\mathbf{D}^j_s Z^{z,i}_t $ is the $j$-th component of
$\mathbf{D}_s Z^{z,i}_t$,  and where we have set $\bj_{s,t}=\bj_{t}\,\bj_{s}^{-1}$.
\end{proposition}
{
\begin{proof}
The integrability of the Jacobian $\bj$ stated in (i) is the main content of \cite[Theorem 6.5]{CLL}.  The fact that $Z_t^{z,i} \in
\mathbb{D}^{\infty}(\tch)$  is proved in \cite[Theorem 1.2]{Inahama}. Finally, we show that relation \eqref{eq:rep-malliavin-with-jacobian} holds. First note that by Theorem \ref {thm:refined_CM_embedding} and \cite[Proposition 1]{CF} we have,
$$\langle\bd Z^z_t, h\rangle_{\tch}=\bd_{h}Z^z_t=\bj_t\int_0^t\bj_s^{-1}V(Z^z_s)d(\mathcal{R}h)_s,\quad h\in\tch.$$
This together with Remark \ref{rmk:tch-identities} (2) implies that the Malliavin derivative $\bd Z^z_t$ can be realized as a function and 
$$\bd_s Z_t^z=\bj_t\,\bj_s^{-1}V(Z^z_s).$$
The proof is thus completed.

\end{proof}
}
\smallskip

%We also record the basic estimate which is an easy consequence of Proposition \ref{interpolation}..

%\begin{proposition}\label{rmk:bnd-norm-H-pvar}
%For any $\eta\geq1$, there exists a constant $c_\eta>0$ such that for $t \in [0,T]$,
%\begin{align*}
%\EE \lc \Vert \bd X^x_t \Vert_{\tch}^{\eta}\rc \leq c_{\eta} \, t^{\eta H}.
%\end{align*}
%\end{proposition}

\section{Upper bounds for the density}
\label{sec:upper-bounds}

The aim of this section is to study upper bounds for the density of the solution to equation~(\ref{eq:sde}). Throughout this section $X$ is a continuous, centered Gaussian process starting at zero with i.i.d.\ components. In addition, we assume the following uniform ellipticity condition on the vector fields. 
\begin{hypothesis} \label{hyp:elliptic}
The vector fields $V_1,\ldots ,V_d$ of equation \eqref{eq:sde} are $C^\infty$-bounded and form a uniformly elliptic system, that is, for some $\lambda>0$,
\begin{equation}\label{eq:hyp-elliptic}
 v^{*} V(x) V^{*}(x) v  \geq \lambda \vert v \vert^2, \qquad \text{for all } v,x \in \R^n,
\end{equation}
where we have set $V=(V_j^i)_{i=1,\ldots ,n; j=1,\ldots d}$.
\end{hypothesis}

We further introduce

\begin{definition}
Let $X$ be a centered $\R^{d}$-valued Gaussian process with covariance $R$. We assume that $X$ satisfies Hypothesis \ref{hyp:mixed-var-R}. Let $\si_{t}$ and $\ka_{t}$ be as in \eqref{eq:def-variance-Xt}, \eqref{eq:def-kappa}. We define the self-similarity parameter $\eta_{t}$ for $t\in(0,T]$ by
\begin{equation}\label{eq:def-eta}
\eta_{t} := \frac{V_{1,\rho}(R;[0,t]^2)}{R(t,t)}=\lp\frac{\ka_{t}}{\si_{t}}\rp^{2}.
\end{equation}
\end{definition}

\begin{remark}\label{rmk:self-simil-coef}
The name \emph{self-similarity parameter} for $\eta_{t}$ stems from the fact that $\eta_{t}$ {does not depend on $t$} whenever the Gaussian process $X$ is self-similar. Hence, $\eta_t$ can be interpreted as quantifying the lack of self-similarity.
%One can interpret the difference $|\eta_{t}-1|$ as a departure from self-similarity.
\end{remark}

With these definitions at hand, we shall prove an upper bound for the density of $X_t$, under the ellipticity assumption \eqref{eq:hyp-elliptic}.

\begin{theorem}\label{thm:upper-bnd-density}
Let $X$ be an $\R^{d}$-valued continuous, centered Gaussian process starting at zero with i.i.d.\ components and covariance  function $R$. Suppose that Hypotheses \ref{hyp:mixed-var-R}, \ref{hyp:correlation-increments-X}, \ref{assumption1} and \ref{hyp:elliptic} are satisfied and let  $\si_{t}, \ka_{t}, \eta_{t}$ be as in  \eqref{eq:def-variance-Xt}, \eqref{eq:def-kappa}, \eqref{eq:def-eta}.
% Furthermore, we assume that $R$ has finite H\"{o}lder-controlled mixed $(1,\rho)$-variation for a given $\rho\in[1,2)$, {that is $X$ fulfills Hypothesis~\ref{hyp:mixed-var-R}}. Recall that the constants $\si_{t}$,  $\ka_{t}$ and $\eta_{t}$ are respectively defined by \eqref{eq:def-variance-Xt}, \eqref{eq:def-kappa} and \eqref{eq:def-eta}. 
 Let $Z^z$ be the solution to \eqref{eq:sde} driven by the Gaussian rough path lift $\bX$ of $X$. 
%In addition, we assume that  $V_1,\ldots ,V_d$ are elements of $\cac_{b}^{\infty}$ and satisfy the  elliptic condition \eqref{eq:hyp-elliptic}. 
Then for all $t\in(0,T]$, the density $p_t$ of $Z_t^z$ satisfies
\begin{equation}\label{eq:exp-bound-irregular}
p_t(y) \leq \frac{c_{1} \eta_{t}^{n(n+2)}}{\ka_{t}^{n}} 
\exp \left(-\frac{\vert y-z \vert^{1+\frac{1}{\rho}}}{c_{2} \, \ka_{t}^{2} }  \right), \; \;  \text{for all } y \in \R^n,
\end{equation}
for some $c_1,c_2>0$.
\end{theorem}

The reminder of this section is devoted to prove Theorem \ref{thm:upper-bnd-density}. Our global strategy is highlighted in Section \ref{sec:global-strategy}, while the main estimates are derived in Sections \ref{sec:tail-estim}, \ref{sec:estim-malliavin} and \ref{sec:estim-malliavin-matrix}.

\subsection{Global strategy}\label{sec:global-strategy}

Our starting point in order to get the upper bound \eqref{eq:exp-bound-irregular} is the following integration by parts type formula. Denote by $C^\infty_p(\mr^n)$ the space of smooth functions $f$ such that $f$ and all of its partial derivatives have at most polynomial growth.

\begin{proposition} \cite[Proposition 2.1.4]{Nu06} \label{intbyparts}
Let $F=(F^1,\ldots ,F^n)$ be a non-degenerate random vector as in Definition \ref{non-deg}. 
Let $G \in \md^{\infty}$ and $\varphi$ be a function in the space $C_p^\infty(\R^n)$. Then for any multi-index $\alpha \in \{1,2,\ldots ,n\}^k$, $k \geq 1$, there exists
an element $H_{\alpha}(F,G) \in \md^{\infty}$ such that
\begin{equation*}
\me[\partial_{\alpha} \varphi(F) G] =\me[\varphi(F) H_{\alpha}(F,G)],
\end{equation*}
Moreover, the elements $H_{\alpha}(F,G)$ are recursively given by
\begin{equation}\label{eq:recursive-H-alpha}
H_{(i)}(F,G) =\sum_{j=1}^n \delta^\diamond\left( G (\gamma^{-1}_F)^{ij} \, \bd F^j \right) 
\quad\text{and}\quad
H_{\alpha}(F,G)= H_{\alpha_k}(F,H_{(\alpha_1,\ldots ,\alpha_{k-1})}(F,G)),
\end{equation}
and for $1 \leq p<q<\infty$ we have
\begin{equation} \label{Holder}
\Vert  H_{\alpha}(F,G) \Vert_p \leq c_{p,q} \Vert \gamma^{-1}_F \, \bd F\Vert^k_{k, 2^{k-1}r}\Vert G\Vert^k_{k, q},
\end{equation}
 where $\frac1p=\frac1q+\frac1r$.
\end{proposition}

As a consequence, one has the following expression for the density of a non-degenerate random vector.

\begin{proposition} \label{density} \cite[Proposition 2.1.5]{Nu06}
Let $F=(F^1,\ldots ,F^n)$ be a non-degenerate random vector as in Definition \ref{non-deg}. Then the density $p_F(y)$ of $F$ belongs to the Schwartz space, and
for any $\sigma \subset \{1,\ldots ,n\}$,
\begin{equation*}
p_F(y)=(-1)^{n-\vert \sigma \vert}\me[{\bf 1}_{\{ F^i>y^i, i \in \sigma, F^i <y^i, i \notin \sigma\}} H_{(1,\ldots ,n)}(F,1)], \; \;  \text{for all } y \in \R^n. 
\end{equation*}
\end{proposition}

According to the above relation applied to $F=Z_t^z$ and $\sigma=\{i \in \{1,\ldots,n\}: y^i \geq z^i\}$, and applying inequality (\ref{Holder}) with $k=n, p=2, r=q=4$, we obtain
the following general upper bound for the density $p_t$ of $Z_t^z$
\begin{equation} \label{bound}
p_t(y) \leq c  \, \mp(\vert Z^z_t-z \vert \geq \vert y-z \vert )^{1/2} \, \Vert \gamma^{-1}_t \Vert^n_{n, 2^{n+2}}  \, \Vert \bd Z_t^z\Vert^n_{n, 2^{n+2}}, \; \;  \text{for all } y \in \R^n,
\end{equation}
where $\gamma_t$ denotes the Malliavin matrix of $Z_t^z$.
In the remainder of the section, we shall bound separately the three terms in the right hand side of \eqref{bound}.

\subsection{Tail estimates} \label{sec:tail-estim}
This section is devoted to  estimating $\mp(\vert Z^z_t-z \vert \geq \vert y-z \vert )$  on the right hand side of \eqref{bound}. Our main result in this direction is the following proposition.

\begin{proposition}\label{prop:exp-moments-rdes}
Let $X$ be an $\R^{d}$-valued continuous, centered Gaussian process with i.i.d.\ components satisfying Hypothesis \ref{hyp:mixed-var-R} for some $\rho\in[1,2)$. Let $\tau\in(0,T]$, $\ka_{\tau}$ be as in \eqref{eq:def-kappa} and $Z^{z}$, $V$ be as in Theorem \ref{thm:upper-bnd-density}.  Then there exists a constant $c_{2}>0$ such that
\begin{equation}\label{eq:concentration-X}
\PP\lp \sup_{t \le \tau} |Z^z_t-z| \ge y \rp \le  
\exp \left(-\frac{\vert y-z \vert^{1+\frac{1}{\rho}}}{c_{2} \, \ka_{\tau}^{2} }  \right),
\end{equation}
for all $y\in\R^{n}$.
\end{proposition}

\begin{proof}
According to Proposition \ref{prop:Gaussian-rough-path}, which can be applied since the process $X$ fulfills Hypothesis \ref{hyp:mixed-var-R}, there is a rough path lift $\bX$ of $X$.
For $p>2\rho$, define the control $\om_{\bX,p}$ by
\begin{equation}\label{eq:def-control-B}
\om_{\bX,p}(s,t)=\|\bX\|_{p-{\rm var};[s,t]}^{p} = \sum_{n\le \lfloor p \rfloor} \| \bX^{\bn} \|_{\frac{p}{n}-{\rm var};[s,t]}^{1/n}.
\end{equation}
Then \cite[Lemma 10.7]{FV-bk} asserts that
\begin{equation}\label{eq:bnd-X-pvar}
\|Z^z\|_{p-{\rm var};[s,t]} \le  c_V \lp \lc \om_{\bX,p}(s,t)\rc^{1/p} \vee \om_{\bX,p}(s,t) \rp.
\end{equation}
In particular, for any $t_{i}<t_{i+1}$ we have
\begin{equation}\label{eq:ineq-davies-increments}
|\der Z^z_{t_{i}t_{i+1}}| \le 
c_V \lp \lc \om_{\bX,p}(t_{i},t_{i+1})\rc^{1/p} \vee \om_{\bX,p}(t_{i},t_{i+1}) \rp .
\end{equation}
Consider now $\alpha\ge 1$ and construct a partition of $\ot$ inductively in the following way: we set $t_0=0$ and
\begin{equation}\label{eq:def-tau-i}
t_{i+1}:= \inf\lcl  u >t_{i} ; \, \|\bX\|^p_{p-{\rm var};[t_{i},u]} \ge \alpha \rcl.
\end{equation}
We then set $N_{\alpha,t,p}=\sup\{n\ge 0;\, t_n < t\}$. Observe that, since we have taken $\alpha\ge 1$, inequality~\eqref{eq:ineq-davies-increments} can be read as $|\der Z_{t_{i}t_{i+1}}| \le  c_V \,\om_{\bX,p}(t_{i},t_{i+1}) = c_V \,\alpha$. Hence
\begin{equation}\label{eq:telescopic-increments-X}
|Z^z_t-z| \le | Z_t^z-Z_{t_{N_{\alpha,t,p}} } |+ \sum_{i=0}^{N_{\alpha,t,p}-1} |\der Z_{t_{i}t_{i+1}}| \le c_V \,\alpha \, (N_{\alpha,t,p}+1).
\end{equation}
By \cite[Theorem 6.3]{CLL} we have
\begin{equation}\label{eq:concentration-N}
\PP \lp N_{\alpha,t,p} +1> n \rp
\lesssim \exp\lp -\frac{c_{p,q,\al} \, n^{\frac{2}{q}} }{\ka_{t}^{2}}\rp,
\end{equation}
where $\ka_{t}$ is as in \eqref{eq:def-kappa} and $q$ is the exponent given in Theorem \ref{thm:refined_CM_embedding} by $\frac{1}{q}=\frac{1}{2\rho}+\frac{1}{2}$. This easily implies
\begin{equation}\label{eq:concentration-X-2}
\PP\lp \sup_{t \le \tau} |Z^z_t-z| \ge \xi \rp
\le \PP \lp c_V \,\alpha \, (N_{\alpha,\tau,p} +1)> \xi \rp
\lesssim \exp\lp -\frac{c_{p,q,\al,V}  \xi^{1+\frac{1}{\rho}}}{\ka_{\tau}^{2}} \rp,
\end{equation}
and thus the claim.
\end{proof}

\subsection{Estimate for Malliavin derivatives} \label{sec:estim-malliavin}
We now proceed to bound the Malliavin derivatives involved in the right hand side of \eqref{bound}. We summarize the results in the following proposition.

\begin{proposition}\label{est Malliavin derivative}
Under the same assumptions as in Proposition \ref{prop:exp-moments-rdes}, for all $m \in \mathbb{N}$ and $p>1$ there exists a positive constant $c_{m,p}$ such that
\begin{align} 
 \label{derivative}
 \Vert  Z_t^z\Vert_{m, p} &\leq c_{m,p} \, \ka_{t},
\end{align}
where $\ka_{t}=V_{1,\rho}(R; [0,t]^2)^{\frac{1}{2}}$ is as in \eqref{eq:def-kappa}.
\end{proposition}

%\begin{lemma}\label{th:DX mp norm}
%Let $H>\frac{1}{4}$. Denote by $X_t^x$ the solution to equation \textnormal{(\ref{eq:sde})}.  One has
%$$\|\bd X_t^x\|_{m,p}\leq c_{m,p} t^{H},
%$$
%for some constant $c_{m,p}>0$. 
%\end{lemma}

\begin{proof}
%\noindent Do we have an embedding like $\tch\hookrightarrow\cac^{p-var}$. Since we are trying to bound $\|\bd^nX\|_{\tch^{\otimes n}}$. Assuming (\ref{eq:high-order-bound2}) and the correct embedding, to show that $\bd^nX$ is bounded in $\|\cdot\|_{\tch^{\otimes n}}$, I guess we have something like

We use a method by Inahama \cite{Inahama} to which we refer for more details. For simplicity, we assume $V_0=0$, and first show \eqref{derivative} for $m=1,2$. The case $V_0 \neq 0$ is treated similarly. Recall that $\mathbf{J}$ is the Jacobian process. 

\noindent
\emph{Step 1: Expression for the Malliavin derivatives.}
Let $\hat{X}=(\hat{X}_1,...,\hat{X}_d)$ be an independent copy of $X$ and consider the $2d$-dimensional Gaussian process $(X, \hat{X}).$ The expectation with respect to $X$ and $\hat{X}$ are respectively denoted by $\me$ and $\hat{\me}$.
Set
$$
\Xi_{t}^{1}:=\sum_{j=1}^{d} \mathbf{J}_t\int_0^t\mathbf{J}^{-1}_sV_{j}(Z_s^z)d\hat{X}_{s}^{j},$$
and
\begin{align*}
\Xi_{t}^{2}:=&
\sum_{j=1}^{d} \mathbf{J}_t\int_0^t\mathbf{J}^{-1}_s
\lcl 
D^2V_{j}(Z_s^z)\lp \Xi_{s}^{1}, \Xi_{s}^{1}\rp dX_s^{j} +2DV_{j}(Z_t^z) \Xi_{s}^{1}, d\hat{X}_s^{j}
\rcl.
\end{align*}
%\hb{Question: wouldn't it be simpler and more consistent to use Proposition \ref{interpolation-general} (i)  here, instead of Inahama's method?}
Then one can show that the following bounds hold true (for more details, see equations (2.8) and (2.9) in \cite{Inahama}, and the discussion after them),
$$\|\bd Z_t^z\|_{\tch\otimes\mr^n} \le C (\hat{\me}\vert\Xi_{t}^{1}\vert^2)^{1/2},$$
$$\|\bd^2 Z_t^z\|_{\tch\otimes\tch\otimes\mr^n}\le C (\hat{\me} \vert \Xi_{t}^{2} \vert^2)^{1/2}.$$

\noindent
\emph{Step 2: Bound for the first order derivative.}
We now estimate $\Xi^{1}$ by using general bounds taken from the theory of rough paths. Namely, let
\begin{align}\label{rough M}
M=(X, \hat{X}, Z^z, \mathbf{J}, \mathbf{J}^{-1}).
\end{align}
Then, $M$  {can be lifted} as a rough path {$\mathbf{M}$} obtained by solving an SDE driven by $(X, \hat{X})$. Hence, it is a $p$-rough path for any $p>2\rho$, where $\rho$ is the exponent appearing in Hypothesis~\ref{hyp:mixed-var-R}. Furthermore, the integral $\int \mathbf{J}^{-1}_sV(Z_s^z)d\hat{X}_s$ is a rough integral of the type $\int f(M) d\mathbf{M}$, where $f$ has polynomial growth. We deduce  that for some $r>0$, the following bound is verified
\begin{equation}\label{b0}
| \delta\Xi_{st}^{1} | \le C (1+\| \mathbf{M}\|_{p-var,[0,T]})^r \| \mathbf{M}\|_{p-var,[s,t]}.
\end{equation}
{We now estimate  $\|\mathbf{M}\|_{p-var,[s,t]}$ appearing in \eqref{b0}. Define
$$Y_t=\sum_{j=1}^d\int_0^t DV_j(Z^z_s)dX_s^j.$$
Then $\tilde{M}=(X,\hat{X}, Z^z, Y)$ can be lifted as a rough path $\tilde{\mathbf{M}}$ by solving an SDE (with $C^\infty$-bounded vector fields) driven by $(X, \hat{X})$. Note that the Jacobian satisfies equation \eqref{eq:jacobian} and  that $\mathbf{J}^{-1}$ satisfies a similar equation
$$
\bj_{t}^{-1} = \id_{n} - \int_0^t \bj_s^{-1} DV_0 (Z^z_s)  \, ds-
\sum_{j=1}^d \int_0^t \bj^{-1}_s DV_j (Z^{z}_s)  \, dX^j_s.
$$
Also recall that we assumed $V_0=0$ throughout our proof. It is then clear that the rough path $\mathbf{M}$ can be obtained by solving an SDE (with linear vector fields) driven by $\tilde{\mathbf{M}}$. Hence, we have the following growth-bound {(cf.\ \cite[inequality (4.10) and Remark 4.12]{CLL})},
\begin{equation}\label{eq:upp-bnd-J-with-B-p-var}
\|\mathbf{M}\|_{p-{\rm var}; [0,t]} \leq C \, \| \mathbf{\tilde{\mathbf{M}}}\|_{p-var,[0,t]}\exp\left({C N_{\alpha,t,p}(\tilde{\mathbf{M}})}\right),
\end{equation}
where $N_{\alpha, t,p}(\tilde{\mathbf{M}})$ is defined in \cite[equation (4.7)]{CLL} and has finite moment to any order by Corollary 3 of \cite{FR}.
Gathering \eqref{b0} and \eqref{eq:upp-bnd-J-with-B-p-var}, together with \cite[Lemma 4]{FR}, we  deduce that
\begin{equation}\label{b1}
| \Xi_{t}^{1} | \le C  \| \mathbf{\tilde{\mathbf{M}}}\|_{p-var,[0,t]}\exp\left({C N_{\alpha,t,p}(\tilde{\mathbf{M}})}\right).
\end{equation}

Furthermore, by standard rough path estimate for SDEs with $C^\infty$-bounded vector fields (cf. \cite[Theorem 10.36]{FV-bk}), we have
$$\|\tilde{\mathbf{M}}\|_{p-{\rm var}; [0,t]}\leq C_V (\|{\mathbf{X}}\|_{p-{\rm var}; [0,t]}+\|{\hat{\mathbf{X}}}\|_{p-{\rm var}; [0,t]})\vee (\|{\mathbf{X}}\|_{p-{\rm var}; [0,t]}+\|{\hat{\mathbf{X}}}\|_{p-{\rm var}; [0,t]})^p.$$
%$$N_{\alpha, t,p}(\tilde{\mathbf{M}})\leq C(N_{\alpha, t,p}(\tilde{\mathbf{M}})+1)$$
%for some constant $C>0$.
%Let $D(t)$ be an arbitrary subdivision of the interval $[0,t]$, and define
%\begin{align*}
%\mathcal{M}_{\alpha,t,p}:=\sup_{D(t)=(t_i); \| \mathbf{X}\|^p_{p-var,[t_i,t_{i+1}]}\leq \alpha}\sum_{i: t_i\in D(t)} \|\mathbf{X}\|^p_{p-var,[t_i,t_{i+1}]}.
%\end{align*}
%Note that both the Jacobian $\bj$ and its inverse $\bj^{-1}$ satisfy a linear RDE driven by $X$. %\begin{align*}
%\|\bj\|_{p-{\rm var}; [0,t]} +\|\bj^{-1}\|_{p-{\rm var}; [0,t]} \leq C \, \| \mathbf{X}\|_{p-var,[0,t]}\exp\left({C \mathcal{M}_{\alpha,t,p}}\right).
%\end{align*}
%In addition, for some constant $c$ (cf.\ \cite[Proposition 4.11]{CLL}), we have $\cm_{\alpha,t,p}\leq c(N_{\alpha,t,p}+1)\alpha.$
%Hence, we obtain a bound for $\Vert \bj \Vert_{p-{\rm var};[0,t]}$ of the form:
We now invoke \cite[Theorem 35-(i) and Corollary 66]{FV},   which asserts that
$$\big\|\| \mathbf{X}\|_{p-var,[0,t]}+\| \hat{\mathbf{X}}\|_{p-var,[0,t]}\big\|_{L^q}
\leq C_q \ka_{t} .
$$
First using H\"{o}lder's inequality in \eqref{b1} and then the estimate above completes the proof of ~\eqref{derivative} for $m=1$. 
}

\noindent
\emph{Step 3: Higher order derivatives.}
In the same way as in Step 2, we estimate $\Xi^{2}$ as a rough integral of the type $\int \phi (M_1) d\mathbf{M}_1$ where $\phi$ has polynomial growth and $M_1$ is the rough path
\[
M_1 =(X, \hat{X}, Z^z, \mathbf{J}, \mathbf{J}^{-1},\Xi^{1})
\]
Arguing as before and using all the previous estimates, we obtain a bound of the same type as~\eqref{b1}
\[
| \Xi_{t}^{2} | \le C  \| \mathbf{\tilde{\mathbf{M}}}\|_{p-var,[0,t]}\exp\left({C N_{\alpha,t,p}(\tilde{\mathbf{M}})}\right).
\]
This easily yields the claim \eqref{derivative} for the case $m=2$. 
Higher order Malliavin derivatives are treated similarly by constructing processes $\Xi^m, m>2$ inductively (see \cite{Inahama}). 
\end{proof}

%{\color{red}
%\begin{remark}\label{Malliavin derivative-Hornamder}
%It is clear from the proof that the above estimate for the Malliavin derivatives remains true even under H\"{o}rmander's %condition.
%\end{remark}
%}
\subsection{Estimates for the Malliavin matrix}\label{sec:estim-malliavin-matrix}   

We next provide an estimate for the inverse of the Malliavin matrix $\gamma_t$ in  \eqref{bound}. 

\begin{proposition}\label{est Malliavin Matrix}
Consider the solution $Z^{z}$ to \eqref{eq:sde} under the same conditions as in Theorem \ref{thm:upper-bnd-density}. For $t\in(0,T]$, let $\ga_{t}$ be its Malliavin matrix defined as in \eqref{malmat}.
Then, for all $m \in \mathbb{N}$ and $p>1$ there exists a constant $c_{m,p}$ such that 
\begin{align} \label{gamma}
\Vert \gamma^{-1}_t \Vert_{m, p} &\leq 
\frac{c_{m,p} \, \eta_{t}^m}{\si_{t}^{2}},
\end{align}
where $\si_{t}$, $\eta_{t}$ are as in relations \eqref{eq:def-variance-Xt} and \eqref{eq:def-eta}.
\end{proposition}

\begin{proof}

Without loss of generality, we will prove (\ref{gamma})  for $0 < t \le 1$. We divide the proof into two steps.

\noindent
\emph{Step 1: case $m=0$.}
Let $C_{t}$ be the matrix defined by
\[
C_t= 
\int_0^t \int_0^t \mathbf{J}_{u}^{-1} V(Z_u^x)V(Z_v^x)^*(\mathbf{J}_{v}^{-1} )^* dR(u, v).
\]
%According to \cite[Proposition 4]{CFV} 
By Remark \ref{representation H norm} and (\ref{eq:rep-malliavin-with-jacobian}), we have $\ga_{t}=\mathbf{J}_{t} C_{t} \mathbf{J}_{t}^{*}$. Therefore the upper bound on  $\|\ga_{t}^{-1}\|_{p}$ can be easily deduced from the following inequality 
\begin{equation}\label{eq:low-bnd-Sigma-t}
y^{*} C_t y \geq M_{t} \si_{t}^{2} \, |y|^{2}, 
\quad \text{for} \quad
y\in\R^{n},
\end{equation}
where $M_{t}$ is a random variable  admitting negative moments of any order (see, e.g.\ \cite[Lemma 2.3.1]{Nu06}).  To this aim, we first notice that
\begin{equation}\label{b2}
y^{*} C_t y = \| f \1_{[0,t]} \|_{\tch}^{2},
\quad \text{with} \quad
f_{u} :=  V(Z_{u}^z)^*(\mathbf{J}_{ u}^{-1} )^* y.
\end{equation}
Furthermore, thanks to the interpolation inequality \eqref{interpolation3}, we have 
\begin{equation}\label{eq:interpolation-H-L2}
\| f \1_{[0,t]} \|_{\tch}^{2}
\ge
\frac{\si_{t}^{2} \|f\|_{\infty;[0,t]}^{2}}{4}
\min\left\{1, \
\frac{c_{X} \, \|f\|_{\infty;[0,t]}^{\frac{\alpha}{\gamma}}}
{\si_{t}^{2} \, \|f\|_{\gamma;[0,t]}^{\frac{\alpha}{\gamma}}}
\right\}.
\end{equation}
Next observe that, due to the uniform ellipticity condition $| V(x) y |^2 \ge \lambda | y |^2$, it is readily checked that
\begin{equation}\label{eq:bnd-f}
|f_v|^2 \ge \la \, |\mathbf{J}_{ v}^{-1} y |^2\ge \la \, \|\mathbf{J}_{ v}\|^{-2} | y |^2.
\end{equation}
Moreover, we have $J_{0}=\id$, which implies that $\sup\{\|J_{v}\|^{-1}; v\in[0,t]\}\ge 1$. Relation \eqref{eq:bnd-f} thus yields %\hb{(I have changed some of the exponents for $|y|$ below, this is to be checked)}:
\begin{equation}\label{b3}
\|f\|_{\infty;[0,t]} \ge \la |y|.
\end{equation}
Plugging \eqref{b3} into \eqref{eq:interpolation-H-L2}, we thus get
\begin{equation*}
\| f \1_{[0,t]} \|_{\tch}^{2}
\ge
\si_{t}^2 M_{t} |y|^{2},
\quad\text{with}\quad
M_{t} = 
\frac{\la^2}{4}
\min\left\{1, \
\frac{c_{X} \, (\la |y|)^{\frac{\alpha}{\gamma}}}
{\si_{t}^{2} \, \|f\|_{\gamma;[0,t]}^{\frac{\alpha}{\gamma}}}
\right\}.
\end{equation*}
According to \eqref{eq:low-bnd-Sigma-t} and \eqref{b2}, it is therefore left to prove $\EE[M_{t}^{-p}]<\infty$ for all $p\ge 1$, {uniformly in $t$ and $y$}. We trivially have
\begin{equation}\label{b4}
M_{t}^{-1} \le 
\frac{4}{\la^2}
\max\left\{1, \
\frac{\si_{t}^{2} \, \|f\|_{\gamma;[0,t]}^{\frac{\alpha}{\gamma}}}
{c_{X} \, (\la |y|)^{\frac{\alpha}{\gamma}}}
\right\} ,
\end{equation}
and by definition of $f$ in \eqref{b2}
\begin{equation*}
\|f\|_{\gamma;[0,t]} \le \|J^{-1} V(Z^{z})\|_{\gamma;[0,t]} \, |y|.
\end{equation*}
Substituting this value in \eqref{b4} yields
\begin{equation}\label{b41}
M_{t}^{-1} \le 
\frac{4}{\la^2}
\max\left\{1, \
\frac{\si_{t}^{2} \, \|J^{-1} V(Z^{z})\|_{\gamma;[0,t]}^{\frac{\alpha}{\gamma}}}
{c_{X} \, \la^{\frac{\alpha}{\gamma}}}
\right\}.
\end{equation}
It is thus readily checked that $M_{t}^{-1}$ admits moments of any order {uniformly in $t$ and $y$}, thanks to the fact that $\|J^{-1} V(Z^{z})\|_{\gamma;[0,t]}$ admits moments of any order. Indeed, similar arguments as used in \cite{CLL} to control the $p$-variation norm of $J^{-1}$ can be used to show that the $\gamma$-H\"{o}lder norm of $J^{-1}$ admits moments of any order. This concludes the proof for $m=0$, namely
\begin{equation}\label{b5}
\|\gamma_t^{-1}\|_p\leq c \, \si_{t}^{-2}.
\end{equation}

\noindent
\emph{Step 2: case $m\ge 1$.}
Now that we have established \eqref{b5}, the case of higher order derivatives follows from more standard considerations. Indeed, applying elementary rules for the derivative of the inverse to $\gamma_t^{-1}$, we get
\begin{align}\label{D gamma}
\bd(\gamma_t^{-1})^{ij}=-\sum_{k,l=1}^d(\gamma_t^{-1})^{ik}(\gamma_t^{-1})^{lj}\bd\gamma_t^{kl}.
\end{align}
Therefore, it is easily seen that, using the definition of $\gamma_t$,
\begin{equation*}
\|\bd(\gamma_t^{-1})^{ij}\|_\tch
\le
c_{d} \lp\|\bd Z_t\|_\tch+\|\bd^2Z_t\|_{\tch^{\otimes 2}}\rp^2 \|\gamma_t^{-1}\|^{2}.
\end{equation*}
Together with \eqref{derivative} and \eqref{b5} this implies
\begin{align*}
\|\bd(\gamma_t^{-1})^{ij}\|_\tch
\le
\frac{c_{d} \, \ka_{t}^{2}}{\si_{t}^{4}} 
=
\frac{c_{d} \, \eta_{t}}{\si_{t}^{2}},
\end{align*}
which yields the claim \eqref{gamma} for $m=1$. Similarly, by using equation (\ref{D gamma}) repeatedly, we  obtain the general case of relation \eqref{gamma}.
\end{proof}

We can now conclude this section by giving a short proof of the main theorem.

\begin{proof}[Proof of Theorem \ref{thm:upper-bnd-density}]
We plug the estimates \eqref{eq:concentration-X}, \eqref{derivative} and \eqref{gamma} into \eqref{bound}. This easily yields the claim \eqref{eq:exp-bound-irregular}.
\end{proof}

\begin{remark} Concerning the dependence of the constants $c_1,c_2$ in \eqref{eq:exp-bound-irregular} on $T$  we note the following:
\noindent
\emph{(i)} 
An analysis of the proof of Proposition \ref{prop:exp-moments-rdes} yields that $c_2$ can be chosen independently of the time horizon $T$.

\noindent
\emph{(ii)} The dependence of $c_1$ on $T$ is less explicit, since it relies on the constant $c_X$ appearing in Hypothesis \eqref{assumption1}, which in turn is intimately linked to the variance of the driving process $X$ (cf.\ e.g.\ Example \ref{example_stat}). In the case of fractional Brownian motion, Hardy-Littlewood's lemma (see e.g \cite[Equation (5.20)]{Nu06}) reveals that $c_{X}$ is bounded from below uniformly in $T$. Assuming that this is the case, an analysis of the derivation of \eqref{derivative} shows that $c_2$ depends on $T$ via $M^{\ka_{T}^{2/(1+1/\rho)}}$ for some $M>1$.
%
%\noindent
%\emph{(ii)} 
%The estimates \eqref{derivative} of the Malliavin derivatives involve a random process $M$ defined by~\eqref{rough M}, which contains the Jacobian $J$. Then the expected values related to $J$ are all bounded as in \eqref{b1}. Taking into account the concentration inequality \eqref{eq:concentration-N}, we let the patient reader check that the terms related to \eqref{derivative} yield a dependence in $T$ of the form $M^{\ka_{T}^{2/(1+1/\rho)}}$ for some $M>1$.
%
%\noindent
%\emph{(iii)} 
%Eventually, one should evaluate the dependence in $T$ for the interpolation inequalities involved in the estimation of $M_{t}^{-1}$ (see \eqref{b41}). Here the only extra dependence on $T$ could only come from the constant $c_{X}$. It is hard to predict this dependence for a general process $X$. However, let us mention that when $X$ is a fractional Brownian motion then Hardy-Littlewood's lemma (see e.g \cite[Equation (5.20)]{Nu06}) asserts that $c_{X}$ can be lower bounded uniformly in $T$.
%
%\noindent
%In conclusion, the constant $c_1$ % in plugging again the estimates \eqref{eq:concentration-X}, \eqref{derivative} and \eqref{gamma} into \eqref{bound}, we get that the constant $c_{1}$ 
%in \eqref{eq:exp-bound-irregular} depends on $T$ via $M^{\ka_{T}^{2/(1+1/\rho)}}$ for some $M>1$. % for a universal constant $c_{3}$. 
\end{remark}

\section{Varadhan estimate}
\label{sec:varadhan}
 Fix a small parameter $\eps\in(0,1]$, and consider the solution $Z_t^\eps$ to the stochastic differential equation 
\begin{align}\label{equ: SDE} 
{Z_t^\eps=z+\int_0^tV_0(Z_s^\eps)ds+\eps\sum_{i=1}^d\int_0^tV_i(Z_s^\eps)dX_s^i,}\quad \forall t\in[0,T],
\end{align}
where, as before, the vector fields $V_0,V_1,\ldots,V_d$ are $C^\infty$-bounded vector fields on $\R^n$. In this section we will work under the same assumptions as in Section \ref{sec:upper-bounds} which are summarized as follows.

\begin{hypothesis}\label{hyp:varadhan}
Let $X$ be an $\R^{d}$-valued continuous, centered Gaussian process starting at zero with i.i.d.\ components and covariance function $R$ %having finite H\"{o}lder-controlled mixed $(1,\rho)$-variation for some $\rho\in[1,2)$, {that is, $X$ fulfills
satisfying Hypothesis~\ref{hyp:mixed-var-R}. We further assume that $X$ satisfies Hypothesis \ref{hyp:correlation-increments-X} and~\ref{assumption1} and that the vector fields $V_1,\ldots ,V_d$ satisfy Hypothesis \ref{hyp:elliptic}. Without loss of generality we choose $T=1$.
\end{hypothesis}

With Hypothesis \ref{hyp:varadhan} at hand, we will describe the asymptotic behavior of the density of $Z_t^\eps$ as $\ep\to 0$. We start by recalling the large deviation setting for rough paths in Section \ref{sec:large-deviation}, and will complete the estimates in Section \ref{sec:asymptotic-varadhan}.

\subsection{Large deviations setting}\label{sec:large-deviation}

Let us first recall that under Hypothesis \ref{hyp:varadhan}, $X$ can be lifted to a p-rough path with $p>2\rho$. According to the general rough path theory (see, e.g., inequality (10.15) and Theorem 15.33 in \cite{FV-bk}), for any positive $\lambda$ and $\delta<2/p$ we have 
\begin{equation}\label{eq:exp-delta-moments}
\me\left[\exp\left( \lambda \sup_{t\in [0,1], \epsilon \in (0,1]}|Z^\eps_t|^\delta\right)\right]<\infty.
\end{equation}
In addition, the Malliavin derivative  and Malliavin matrix  of $Z_1^\eps$ can be controlled  using the same arguments as in the previous section. More precisely,  replacing the $V_i$'s with $\eps V_i$'s in the proof of Propositions \ref{est Malliavin derivative} and \ref{est Malliavin Matrix}, we have
%\begin{lemma}\label{th: Malliavin est}
%Assume Hypothesis \ref{}, we have
\begin{align}
&\sup_{\eps\in(0,1]}\|Z_1^\eps\|_{k,r}<\infty,\quad \mathrm{\ for\ each\ } k\geq 1 \mathrm{\ and }\  r\geq 1;\label{M derivative}\\
&\ \|\gamma_{Z_1^\eps}^{-1}\|_r\leq c_r \eps^{-2},\quad \mathrm{for\ any\ } r\geq 1,\label{M matrix}
\end{align}
where $\gamma_{Z_1^\eps}$ is the Malliavin matrix of $Z^\eps_1$.

Denote by $\bj^\eps$ the Jacobian  of $Z^\eps$. Similar to \eqref{eq:jacobian}, the process $\bj^\eps$ is the unique solution to the linear equation
\begin{equation*}
{\bj_{t}^\eps = \id_{n} +\int_0^tDV_0(Z_s^\eps)\bj_s^\eps ds+
\eps\sum_{j=1}^d \int_0^t DV_j (Z^{\eps}_s) \, \bj_{s}^\eps \, dX^j_s.}
\end{equation*}
Its moments are uniformly bounded (in $\eps\in(0,1]$) in the next proposition.
\begin{proposition}\label{prop:moments-jacobian}
For any $\eta\ge 1$, there exists a finite constant $c_\eta$ such that the Jacobian $\bj^\eps$ satisfies 
\begin{equation}
\sup_{\eps\in(0,1]} \EE\lc  \Vert \bj^\eps \Vert^{\eta}_{p-{\rm var}; [0,1]} \rc = c_\eta.
\end{equation}
\end{proposition}

\begin{proof}
When $\eps=1$, the integrability of $\bj^\eps$ is proved  in \cite{CLL}, and has been recalled in Proposition \ref{prop:deriv-sde} above.  It can be checked that the estimates in \cite{CLL} only depends on the supremum norm of the vector fields and their derivatives. In the present case, the vector fields $\eps V_i$ in equation (\ref{equ: SDE}) are uniformly bounded in $\eps\in (0,1]$ together with their derivatives. Hence the uniform integrability of $\bj^\eps$ (in $\eps$) follows. 
\end{proof}

In order to state a  large deviation type result, let us introduce the so-called skeleton of equation \eqref{equ: SDE}, that is, we introduce the map %. Namely, we consider an application 
$\Phi: \ch\to \mathcal{C}([0,1],\mathbb{R}^{n})$ associating to each $h\in  \ch$ the unique solution of the ordinary differential equation
\begin{align}\label{phi}
{\Phi_t(h)=z+\int_0^tV_0(\Phi_s(h))ds+\sum_{i=1}^d\int_0^tV_i(\Phi_s(h))dh_s^i.}
\end{align}
By the embedding Theorem \ref{thm:refined_CM_embedding}, for each $h\in\ch$, the above equation can be understood in Young sense. In particular, it follows that there is a unique solution $\Phi_\cdot(h)$.  Moreover, $\Phi_t$ is a differentiable mapping from $\ch$ to the space $\mathcal{C}([0,1],\mathbb{R}^{n})$. We let $\gamma_{\Phi_1(h)}$ be the deterministic Malliavin matrix
 of $\Phi_1(h)$,  that is,
\begin{equation}\label{eq:def-deterministic-malliavin}
\gamma^{ij}_{\Phi_1(h)}=\langle \bd \Phi_1^i(h), \bd\Phi_1^j(h)\rangle_\tch.
\end{equation}
Along the same lines, we introduce the Jacobian $J(h)$ of equation (\ref{phi}), that is the unique solution of the following equation
\begin{equation}\label{eq:def-jacob-h}
J_t(h)=\id_{n}+\sum_{i}\int_0^tDV_i(\Phi_s(h))J_s(h)dh_s^i+\int_0^tDV_0(\Phi_s(h))J_s(h)ds.
\end{equation}

\begin{remark}\label{convention phi} For a geometric p-rough path $ \mathbf{x}$, it is sometimes convenient to write $\Phi(\mathbf{x})$ obtained by solving (\ref{phi}) with $h$ replaced with $\mathbf{x}$. By the general theory of rough path, $\Phi$ is a continuous function of $\mathbf{x}$ in the p-variation topology. We will use this notation without further mention when there is no confusion. 
\end{remark}

\begin{remark}\label{convention x+h} Let $X$ be an $\mr^d$-valued Gaussian process satisfying Hypothesis \ref{hyp:varadhan} and let $h\in\ch$ be an element of the Cameron-Martin space of $X$. We use the notation $\mathbf{X}+h$ to denote lift of $X+h$ to a $p$-rough path. This construction is made possible by the embedding in Theorem~\ref{thm:refined_CM_embedding} and Young's pairing. We direct the readers to Section 9.4 of \cite{FV-bk} for more details.
\end{remark}
{
The following lemma will be needed later. 
\begin{lemma}For each $h\in\ch$, we have
\begin{align}\label{smoothness in eps}
\lim_{\eps\downarrow0}\frac{1}{\eps}\left(\Phi_t(\eps {\bf{X}}+h)-\Phi_t(h)\right)=G_t(h),
\end{align}
in the topology of $\md^\infty$, and $G_t(h)$ satisfies an SDE of the form
\begin{multline}\label{equ G}
G_t(h)=\int_0^tDV_0(\Phi_s(h))G_s(h)ds+\sum_{i=1}^d\int_0^tDV_i(\Phi_s(h))G_s(h)dh^i_s \\ 
+\sum_{i=1}^d\int_0^tV_i(\Phi_s(h))dX^i_s.
\end{multline}
\end{lemma}
\begin{proof}
Note that $\Phi_t(\eps{\bf{X}}+h)$ satisfies the following rough SDE
\begin{align}\label{equation for phi ep}{\Phi_t(\eps{\bf{X}}+h)=z+\int_0^tV_0(\Phi_s(\eps{\bf{X}}+h))ds+\sum_{i=1}^d\int_0^tV_i(\Phi_s(\eps{\bf{X}}+h))d(\eps{\bf{X}}^i+h_s^i).}\end{align}
By standard path-wise estimates, $\Phi_t(\eps{\bf{X}}+h)$ is smooth in $\eps$ and its derivatives satisfy a rough SDE obtain by formally differentiating \eqref{equation for phi ep} on both sides (see, e.g., \cite[Proposition 11.4]{FV-bk}). In particular, at $\eps=0$, we have
\begin{align*}
\lim_{\eps\downarrow0}\frac{1}{\eps}\left(\Phi_t(\eps {\bf{X}}+h)-\Phi_t(h)\right)=G_t(h),
\end{align*}
where $G_t(h)$ satisfies the equation \eqref{equ G}. The fact that the above convergence takes place in $\mathbb{D}^\infty$ follows the same lines of the proof of Proposition 2.14 in \cite{BO-Varadhan}.
\end{proof}
}
Comparing equations \eqref{equ G} and \eqref{eq:def-jacob-h}, an elementary 
variational principle argument reveals that
\begin{align}\label{explicit G}
G_t(h)=J_t(h)\int_0^t (J_s(h))^{-1}V_i(\Phi_s(h))dX^i_s,
\end{align}
which implies that $G_t(h)$ is a centered Gaussian random variable. Moreover, starting from equation (\ref{explicit G}), some easy computations show that the Malliavin derivative of $G_t(h)$ and the deterministic Malliavin derivative of $\Phi$ at $h$ coincide.
Hence, the covariance matrix of $G_{1}(h)$ is the deterministic Malliavin matrix $\gamma_{\Phi_1(h)}$.

As a last preliminary step we recall the large deviation principle for stochastic differential equations driven by Gaussian rough path, which is the basis for Varadhan type estimates and is standard in rough paths theory (see \cite[Section 19.4]{FV-bk}).

\begin{theorem}\label{th: LDP}
Let $\Phi$ be as in (\ref{phi}), $Z_1^\eps$ be the solution to equation (\ref{equ: SDE}) and set
$$I(y):=\inf_{\Phi_1(h)=y}\frac{1}{2}\|h\|_{\ch}^2\quad\forall y\in\mr^n. $$
Then $Z_1^\eps$ satisfies a large deviation principle with rate function $I(y)$.
\end{theorem}

\begin{proof}
First, it is known  (see, e.g., \cite[Theorem 15.55]{FV-bk}) that $\eps\mathbf{X}$, as a p-rough path, satisfies a large deviation principle in the $p$-variation topology with good rate function given by
\begin{align*}
\mathrm{Rt}(h)=\left\{\begin{array}{ll}\frac{1}{2}\|h\|^{2}_{\ch}\ \mathrm{if}\ h\in\ch\\ +\infty\quad \mathrm{otherwise}.
\end{array}
\right.
\end{align*}
Moreover, by Remark \ref{convention phi}, $\Phi_1(\mathbf{x})$ is continuous function of $\mathbf{x}$ in $p$-variation topology. Since $Z_1^\eps=\Phi_1(\eps \mathbf{X})$ the result follows from the contraction principle.\end{proof}

\subsection{Asymptotic behavior of the density}\label{sec:asymptotic-varadhan}

Recall that the skeleton $\Phi$ is defined by \eqref{phi}. Our density estimates will involve a ``distance''  which depends on $\Phi$ as follows
\begin{equation}\label{eq:def-distance-varadhan}
d^2(y)=I(y)=\inf_{\Phi_1(h)=y}\frac{1}{2}\|h\|_{\ch}^2,\quad\mathrm{and}\quad d^2_R(y)=\inf_{\Phi_1(h)=y, \det\gamma_{\Phi_1(h)}>0}\frac{1}{2}\|h\|_{\ch}^2.
\end{equation}
{When \eqref{equ: SDE} has no drift term and is driven by a standard Brownian motion, it is shown in \cite[Theorem 1.1]{Leandre PTRF87} that under strong H\"{o}rmander conditions the above two distances are the same.}
Interestingly enough, the two distances $d$ and $d_{R}$ always coincide under the ellipticity assumptions (even with the presence of a drift).
\begin{lemma}\label{equiv distance}
Assume that Hypothesis \ref{hyp:varadhan} is satisfied. Then we have $d^2(y)=d^2_R(y)$ for every $y\in\mr^n$.

\end{lemma}
\begin{proof}
The claimed identity is mainly due to the uniform ellipticity of the vector fields $V_i's$. Indeed, pick any $h\in\ch$ such that $\Phi_1(h)=y$. Recall that $J(h)$ is the Jacobian of the deterministic equation (\ref{phi}) and $\gamma_{\Phi_1(h)}$ is the deterministic Malliavin matrix of $\Phi$ at $h$.
{Similarly to \eqref{eq:rep-malliavin-with-jacobian} we have}
$$\mathbf{D}^k_s\Phi_1(h)=J_1(h)(J_s(h))^{-1}V_k(\Phi_s(h)).$$
%Recall the deterministic Malliavin matrix of $\Phi$ at $h\in\msh$ is
%$$\gamma^{ij}_{\Phi_1(h)}=\langle \mathbf{D}\Phi^i_1(h),\mathbf{D}\Phi^j_1(h)\rangle_{\tch}.$$
{Therefore, owing to the definition \eqref{eq:def-deterministic-malliavin} of the Malliavin matrix, we get the following identity for all $x\in\mr^n$}
\begin{align*}
\sum_{ij}x_i\gamma^{ij}_{\Phi_1(h)}x_j&=\sum_{k}\bigg\|\sum_{i}x_i(\mathbf{D}^k\Phi_1(h))^i\bigg\|_{\tch}^2\\
&=\int_0^t\int_0^t \left\langle x^TJ_{u1}(h)V(\Phi_{u}(h))\,,\,x^TJ_{v1}(h)V(\Phi_v(h)\right\rangle dR(u,v).
\end{align*}
Let us now define a function $f$ by
$$f_u=x^TJ_{u1}(h)V(\Phi_u(h)).$$
Under the same assumptions as in Proposition \ref{th: interpolation}, which are satisfied due to Hypothesis~\ref{hyp:varadhan}, we have the interpolation inequality (see relation \eqref{interpolation3})
\begin{align*}
&\int_0^1 \int_0^1 \langle f_{u}, f_{v} \rangle dR(u,v)  
%\nonumber\\&\quad
\ge \frac{1}{4}\sigma_1^2\|f\|_{\infty;[0,1]}^2\min\left\{1,\frac{2\left(\frac{c_X}{2}\right)^{\frac{2\gamma+\alpha}{4\gamma}}}{\sigma_1}\frac{\|f\|_{\infty;[0,1]}^{\frac{\alpha}{2\gamma}}}{(1+\|f\|_{\gamma;[0,1]}^{\frac{\alpha}{2\gamma}})}\right\}^2.
\end{align*}
Furthermore,  the uniform ellipticity condition implies that for any $x\not=0$,
$$\|f\|_{\infty;[0,1]}>0.$$
Therefore, the deterministic Malliavin matrix $\gamma_{\Phi_1(h)}$ is non-degenerate at $h$. In conclusion, for any $h\in\ch$ such that $\Phi_1(h)=y$ we have $\det\gamma_{\Phi_1(h)}>0$ and thus $d_R(y)\equiv d(y)$.
\end{proof}

Now we can state  the  main result of this section, giving the logarithmic asymptotic behavior of the density as $\ep\to 0$.
\begin{theorem}\label{th: main result}
Let $Z^{\ep}$ be the process defined by \eqref{equ: SDE}, and denote by $p_\eps(y)$ the density of $Z_1^\eps$. Due to Hypothesis \ref{hyp:varadhan},  we have
\begin{align*}
\lim_{\eps\downarrow0}\eps^2\log p_\eps(y)= -d^2(y),
\end{align*}
where $d$ is the function defined by \eqref{eq:def-distance-varadhan}.
\end{theorem}

\begin{proof}
With the previous estimates in hand, the proof is similar to the one of \cite[Theorem 3.2]{BO-Varadhan}. For the reader's convenience, we give some details below. Let us divide the proof in two steps.

\noindent
\emph{Step 1: Lower bound.} We shall prove that
\begin{align}\label{main claim 1}
\liminf_{\eps\downarrow0}\eps^2\log p_\eps(y)\geq -d^2_R(y).
\end{align}
To this aim, fix $y\in\mr^n$. We only need to show \eqref{main claim 1} for  $d^2_R(y)<\infty$, since the statement is trivial whenever $d^2_R(y)=\infty$. Next fix an arbitrary $\eta>0$ and let $h\in\ch$ be such that $\Phi_1(h)=y$ and $\|h\|^2_{\ch}\leq d^2_R(y)+\eta$. Let $f\in C_0^\infty(\mr^n).$ By Cameron-Martin's theorem for the Gaussian process $X$, it is readily checked that
$$\me\lc f(Z^\eps_1)\rc=e^{-\frac{\|h\|_{\ch}^2}{2\eps^2}} \, 
\me \lc f(\Phi_1(\eps X+h))e^{-\frac{X(h)}{\eps}}\rc,$$ 
where  $X(h)$ denotes the Wiener integral of $h$ with respect to $X$ introduced in Section~\ref{sec:wiener-space-general}. We now proceed by means of a truncation argument: consider a function $\chi\in C^\infty(\mr)$, satisfying  $0\leq \chi\leq 1$, such that $\chi(t)=0$ if $t\not\in[-2\eta, 2\eta]$, and $\chi(t)=1$ if $t\in[-\eta,\eta]$. Then, if $f\geq 0$, we have
$$\me\lc f(Z^\eps_1)\rc
\geq e^{-\frac{\|h\|^2_{\ch}+4\eta}{2\eps^2}}\, 
\me\lc \chi(\eps X(h))f(\Phi_1(\eps X+h))\rc.
$$
Hence, by means of an approximation argument applying the above estimate to $f=\delta_{y}$, we obtain
\begin{align}\label{lower bound claim1}
\eps^2\log p_\eps(y)\geq 
-\left(\frac{1}{2}\|h\|_{\ch}^2+2\eta\right)
+\eps^2\log\me\big[\chi(\eps X(h))\delta_y(\Phi_1(\eps X+h))\big].
\end{align}
Indeed, for any non-degenerate random vector $F$, the distribution on Wiener's space $\delta_y(F)$ is an element in $\md^{-\infty}$, the dual of $\md^{\infty}$. The expression $\me [\delta_y(F)G]$ can thus be interpreted as the coupling $\langle \delta_y(F), G\rangle$ for any $G\in\md^{\infty}$ (see \cite[Section 2.1.5]{Nu06}).

Let us now bound the right hand side of equation \eqref{lower bound claim1}. Owing to the fact that $\Phi_1(h)=y$ and thanks to the scaling properties of the Dirac distribution, it is easily seen that
$$
\me\big(\chi(\eps X(h))\delta_y(\Phi_1(\eps X+h))\big)=\eps^{-n}\me\left(\chi(\eps X(h))\delta_0\left(\frac{\Phi_1(\eps X+h)-\Phi_1(h)}{\eps}\right)\right).
$$
In addition, according to the definition \eqref{smoothness in eps}, we have
$$
\lim_{\eps\downarrow 0}\frac{\Phi_1(\eps X+h)-\Phi_1(h)}{\eps}=G_1(h),
$$
and recall that we have established, thanks to \eqref{explicit G}, that $G_1(h)$
is an $n$-dimensional random vector in the first Wiener chaos with variance $\gamma_{\Phi_1(h)}>0$. Hence, $G_1(h)$ is non-degenerate and integrating by parts combined with standard arguments from Malliavin calculus yields %\hb{(Question: is relation \eqref{b6} obtained under Hypothesis \ref{hyp:mixed-var-R} only, like in Proposition \ref{est Malliavin derivative}?)}:
\begin{equation}\label{b6}
\lim_{\eps\downarrow0}\me\left[\chi(\eps X(h))\delta_0\left(\frac{\Phi_1(\eps X+h)-\Phi_1(h)}{\eps}\right)\right]
=
\me\lc\delta_0(G_1(h))\rc.
\end{equation}
In particular, we get
$$\lim_{\eps\downarrow0}\eps^2\log\me\big(\chi(\eps X(h))\delta_y(\Phi_1(\eps X+h))\big)=0.$$
Plugging this information in (\ref{lower bound claim1})  and letting $\eps\downarrow0$ we end up with
$$
\liminf_{\eps\downarrow0}\eps^2\log p_{\eps}(y)\geq
-\lp \frac{1}{2}\|h\|^2_{\ch}+2\eta\rp
\geq -\lp d^2_R(y)+3\eta\rp.
$$
Since $\eta>0$ is arbitrary this yields (\ref{main claim 1}). {At this point we can notice that we have chosen $h$ such that $\|h\|^2_{\ch}\leq d^2_R(y)+\eta$  in order to get a non degenerate random variable $G_{1}(h)$ in \eqref{b6}.}

\noindent
\emph{Step 2: Upper bound.}
Next, we show that 
\begin{align}\label{main claim 2}
\limsup_{\eps\downarrow0}\eps^2\log p_\eps(y)\leq -d^2(y).
\end{align}
Towards this aim, fix a point $y\in\mr^n$ and consider a function $\chi\in C_0^\infty(\mr^n), 0\leq\chi\leq1$ such that $\chi$ is equal to one in a neighborhood of $y$. The density of $Z_1^\eps$ at point $y$ is given by
$$
p_\eps(y)=\me\lc\chi(Z_1^\eps)\delta_y(Z_1^\eps)\rc.
$$
Next integrate the above expression by parts in the sense of Malliavin calculus thanks to Proposition \ref{intbyparts}. This yields
\begin{align*}
\me[\chi(Z_1^\eps)\delta_y(Z_1^\eps)]=&
\me\left[\mathbf{1}_{\{Z_1^\eps>y\}}H_{(1,2,...,n)}(Z_1^\eps,\chi(Z_1^\eps))\right]\\
\leq&\me\lc |H_{(1,2,...,n)}(Z_1^\eps,\chi(Z_1^\eps))| \rc\\
=&\me\big[|H_{(1,2,...,n)}(Z_1^\eps,\chi(Z_1^\eps))|\mathbf{1}_{\{Z_1^\eps\in \mathrm{supp}\chi\}}\big]\\
\leq&\mp(Z_1^\eps\in\mathrm{supp}\chi)^\frac{1}{q}\|H_{(1,..,n)}(Z_1^\eps,\chi(Z_1^\eps))\|_p,
\end{align*}
where $\frac{1}{p}+\frac{1}{q}=1$. Furthermore, relation \eqref{Holder} and an application of H\"{o}lder's inequality (see, e.g., \cite[Proposition 1.5.6]{Nu06}) gives 
{$$
\|H_{(1,...,n)}(Z_1^\eps,\chi(Z_1^\eps))\|_p\leq c_{p,q}\|\gamma_{Z_1^\eps}^{-1}\|_\beta^m\|\bd Z_1^\eps\|_{n,\gamma}^r\|\chi(Z_1^\eps)\|^n_{n,q},
$$}
for some constants $\beta, \gamma>0$ and integers $m, r$. Thus, invoking the estimates  (\ref{M derivative}) and (\ref{M matrix}), we obtain
$$\lim_{\eps\downarrow0}\eps^2\log\|H_{(1,...,n)}(Z_1^\eps,\chi(Z_1^\eps))\|_p=0.$$
Finally the large deviation principle for $Z_1^\eps$ recalled in Theorem \ref{th: LDP} ensures that for small $\eps$ we have
$$\mp(Z_1^\eps\in\mathrm{supp}\chi)^\frac{1}{q}\leq e^{-\frac{1}{q\eps^2}(\inf_{z\in\mathrm{supp}\chi}d^2(z)+o(1))}.$$
Since $q$ can be chosen arbitrarily close to 1 and $\mathrm{supp}(\chi)$ can be taken arbitrarily close to $y$, the proof of (\ref{main claim 2}) is now easily concluded {thanks to the lower semi-continuity of $d$.}

Combining Lemma \ref{equiv distance},  (\ref{main claim 1}) and (\ref{main claim 2}), the proof of Theorem \ref{th: main result} is thus completed. 
\end{proof}

{
\begin{remark}
It is clear from the proof of Theorem \ref{th: main result} that the key to establish a Varadhan estimate is to have some quantitative control of the Malliavin derivative and Malliavin matrix of $Z^\eps$. More precisely,
\begin{itemize}
\item[(i)] $\sup_{\eps\in(0,1]}\|Z_1^\eps\|_{k,r}<\infty$, for each $k\geq 1$ and $r\geq1$;  and
\item[(ii)]
for any $r\geq1$, $\|\gamma^{-1}_{Z_1^\eps}\|_r\leq c_r\eps^{-\mu}, \ \mathrm{for\ some}\ \mu>0.$
\end{itemize}
While (i) is generally true for any $C^\infty$-bounded vector fields, the estimate in (ii) needs some non-degeneracy condition on $V$. In this paper, we have restricted our analysis to the elliptic case of Hypothesis \ref{hyp:elliptic} for sake of simplicity. However, one way to extend our results to a H\"{o}rmander type situation would be the following: along the same lines as in \cite{CHLT}, carefully track the dependence on $V$ in order to show that the bound in (ii) for the Malliavin matrix still holds. This step should be enough to prove that the Varadhan estimate is valid. However,  it is worth pointing out that we do not in general have
$$d(y)=d_R(y)$$
when the vector fields are not elliptic. Hence, we expect the corresponding Varadhan estimate under H\"{o}rmander's condition to be read as:
$$-d^2_R(y)\leq\liminf_{\eps\downarrow0}\eps^2\log p_\eps(y)\leq \limsup_{\eps\downarrow0}\eps^2\log p_\eps(y)\leq -d^2(y). 
$$
Also notice that the hypoelliptic situation has been handled when $X$ is a fractional Brownian motion in \cite{BO}.
\end{remark}
}

\section{Applications}

Our main results, Theorem \ref{thm:upper-bnd-density} and Theorem \ref{th: main result} rely on Hypothesis \ref{hyp:mixed-var-R}, \ref{hyp:correlation-increments-X} and~\ref{assumption1}. Let us also recall that the density bound \eqref{eq:exp-bound-irregular} involves a coefficient $\eta$ defined by \eqref{eq:def-eta}. In this section we provide explicit examples of Gaussian processes satisfying the aforementioned assumptions and give estimates for $\eta$ as a function of $t$. 

\begin{remark}\label{about initial value}
The interpolation inequalities in Proposition \ref{interpolation-general}  and Proposition \ref{th: interpolation} rely on an integral representation for the Cameron-Martin norm related to $X$ (see relation (\ref{rep H norm})), which is satisfied for Gaussian processes starting at zero. We note that this is not a restriction in applications, since the RDE (\ref{eq:rde-intro}) driven by $X$ is the same as the one driven by $\tilde{X}=\{\tilde{X}_t=X_t-X_0, t\geq0\}$. Moreover, one easily checks that if $X$ satisfies Hypotheses \ref{hyp:mixed-var-R}, \ref{hyp:correlation-increments-X} and~\ref{assumption1}, then so does $\tilde{X}$.
\end{remark}

\begin{remark}\label{rem:correlation} 
Suppose that $X_t$ is a continuous, centered real-valued Gaussian processes with covariance $R$. Then
\begin{enumerate}
  \item[(i)] If $\partial_{ab}^{2}R \le0$  in the sense of distributions, then Hypothesis \ref{hyp:correlation-increments-X}, (i) is satisfied.
  \item[(ii)] If $\sigma^2_{s,t}=F(|t-s|)$ for some continuous, non-decreasing function $F$ then  Hypothesis \ref{hyp:correlation-increments-X}, (ii) is satisfied.
  \item[(iii)] If $X$ starts at zero, satisfies Hypothesis \ref{hyp:correlation-increments-X}, (i) and $\partial_{a}R(a,b)\ge0$ for $a<b$ in the sense of distributions, then Hypothesis \ref{hyp:correlation-increments-X}, (ii) is satisfied.
\end{enumerate}
\end{remark}
\begin{proof}
  We first note that (i) is proved in \cite[Lemma 2.20]{FGGR13} and (iii) follows from \cite[Section 4.2.1]{CHLT}. For (ii): We have
    \begin{align*}
      2R_{uv}^{st} 
      &= \sigma^2_{s,v}-\sigma^2_{s,u}+\sigma^2_{u,t}-\sigma^2_{v,t} \\
      &= F(|v-s|)-F(|u-s|)+F(|t-u|)-F(|t-v|).
    \end{align*} 
    Since $F$ is non-decreasing this implies, for $s\le u\le v\le t$, $2R_{uv}^{st} \ge0$.
\end{proof}

With this remark in mind, we are now ready to provide a series of examples to which the results of Sections \ref{sec:upper-bounds} and \ref{sec:varadhan} apply.

\begin{example} Let $B^{H}$ be a fractional Brownian motion with Hurst parameter $H\in(0,1)$. As mentioned in Remark \ref{rmk:self-simil-coef}, in this case $\eta_{t}$ does not depend on $t$ due to the self-similarity of $B^{H}$. It is also shown in \cite{CHLT} that Hypothesis \ref{hyp:correlation-increments-X} and~\ref{assumption1} are satisfied whenever $H\in(\frac{1}{4},\frac{1}{2})$. In \cite[Example 2.8]{FGGR13} it is proved that $B^H$ has H\"{o}lder-controlled mixed $(1,\rho)$-variation and thus Hypothesis \ref{hyp:mixed-var-R} is satisfied. \end{example}

\begin{example} \label{example_stat}Let $X$ be a $d$-dimensional centred Gaussian process with i.i.d.~components, such that the coefficient $\sigma^2_{s,t}$ defined by \eqref{eq:def-variance-Xt} satisfies the following relation
\[
\sigma_{s,t}^{2}=F\bigl(\vert t-s\vert\bigr)\geq0,
\]
for some non-negative, concave function $F$ satisfying $F(0)=0$ and
\begin{equation}\label{eq:F'_lowerbd}
  \inf_{s\in[0,T]}F_-'(s)>0,
\end{equation}
{where $F_-'$ denoted the left-hand derivative of the concave function $F$.}

We note that if $F$ is not identically equal to zero, then $F(0)=0$, $F\ge0$ and concavity imply that \eqref{eq:F'_lowerbd} is satisfied for some $T>0$. In addition, we assume that 
\begin{equation}\label{eq:assp_F}
C_{1}t^{\frac{1}{\rho}}\le F(t)\le C_{2}t^{\frac{1}{\rho}}\quad\forall t\in[0,T],
\end{equation}
for some $\rho \in [1,2)$, $C_{1},C_{2}>0$. Since {$2R(s,t)=-F(|t-s|)+F(t)+F(s)$, }  concavity of $F$ and the fact that $F$ is increasing imply Hypothesis \ref{hyp:correlation-increments-X}, due to Remark \ref{rem:correlation}. It is readily checked from \cite[Example 2.9]{FGGR13} that {under assumption \eqref{eq:assp_F} we have}
\[
V_{1,\rho}\bigl(R;[s,t]^{2}\bigr)\leq C\vert t-s\vert^{1/\rho}
\]
for some constant $C>0$ and thus $X$ has H\"{o}lder-controlled mixed $(1,\rho)$-variation. Recalling that $\si^2_{t}:=\si^2_{0,t}$, {invoking \eqref{eq:assp_F} again} we obtain 
\[
\eta_{t}=\frac{V_{1,\rho}\bigl(R;[0,t]^{2}\bigr)}{\si_{t}^{2}}\le C.
\]
In particular, $\eta$ is bounded on $[0,T]$. Finally, from \cite[Theorem 6.1]{FGGR13} we have that Hypothesis~\ref{assumption1} is satisfied with $\alpha=1$. \end{example} 

\begin{example} Let $X=B^{H_{1}}+B^{H_{2}}$ be a sum of two independent fBm with Hurst parameters $H_{1},H_{2}\le1/2$. Then
\[
\sigma_{s,t}^{2}=|t-s|^{2H_{1}}+|t-s|^{2H_{2}}=:F\bigl(\vert t-s\vert\bigr)
\]
and the previous example applies.  \end{example}

\begin{example} \label{examplebifbm} Consider a \textit{bifractional Brownian motion} (cf., e.g., \cite{HV03,RT06,KRT07}), that is, a centered Gaussian process $B^{H,K}$ on $[0,T]$ with covariance function given by\footnote{As pointed out, for example, in \cite{KRT07} this process does not fit in the Volterra framework.} 
\begin{eqnarray*}
 &  & R(s,t)=\frac{1}{2^{K}}\bigl(\bigl(s^{2H}+t^{2H}\bigr)^{K}-\vert t-s\vert^{2HK}\bigr),
\end{eqnarray*}
for some $H\in(0,1)$ and $K\in(0,1]$ such that $HK\le 1/2$. Since $B^{H,K}$ is a self-similar process with index $HK$, the coefficient $\eta$ does not depend on $t$. Hypothesis \ref{hyp:correlation-increments-X} and the fact that $R$ admits a H\"{o}lder-controlled mixed $(1,\rho)$-variation, i.e.\ Hypothesis \ref{rem:correlation}, have been verified in \cite[Example 2.12]{FGGR13}. In order to check Hypothesis~\ref{assumption1} we recall from \cite[equation (6.2)]{FGGR13}, using Hypothesis \ref{hyp:correlation-increments-X}, that
\begin{align*}
  2 \mathrm{Var}(X_{s,t}|\mathcal{F}_{0,s}\vee\mathcal{F}_{t,T}) 
  &\ge 2R\left(\begin{array}{cc}
  0 & T\\
  s & t
  \end{array}\right).
\end{align*}
Hence,
\begin{align*}
  2 \mathrm{Var}(X_{s,t}|\mathcal{F}_{0,s}\vee\mathcal{F}_{t,T}) 
  &\ge2\me{(X_T-X_0)(X_t-X_s)}=2(R(T,t)-R(T,s))\\
  &=2^{1-K}\big((t^{2H}+T^{2H})^K-|t-T|^{2HK})-((s^{2H}+T^{2H})^K-|s-T|^{2HK})\big)\\
  &\ge 2^{1-K}(|s-T|^{2HK}-|t-T|^{2HK})\\
  &\ge C(T)|t-s|,
  %\sigma^2(0,t)-\sigma^2(0,s)+\sigma^2(s,T)-\sigma^2(t,T).
\end{align*}
which implies Hypothesis~\ref{assumption1}.
\end{example}

\begin{example}\label{exrfsnon-stationary} Consider a random Fourier series\footnote{We may ignore the (constant, random) zero-mode in the series since we are only interested in properties of the increments of the process.} 
\[
\Psi(t)=\sum_{k=1}^{\infty}\alpha_{k}Y^{k}\sin(kt)+\alpha_{-k}Y^{-k}\cos(kt),\qquad t\in[0,2\pi],
\]
with zero-mean, independent Gaussians $\{Y^{k}; \, k\in\Z\}$ with unit variance. Then the covariance $R$ can be computed in an elementary way 
\begin{eqnarray}
R(s,t) & = & \sum_{k=1}^{\infty}\alpha_{k}^{2}\sin(ks)\sin(kt)+\alpha_{-k}^{2}\cos(ks)\cos(kt)\label{eqnfourierseriesdecomp}\\
 & = & \frac{1}{2} \sum_{k=1}^{\infty}(\alpha_{k}^{2}+\alpha_{-k}^{2})\cos(k(t-s))+(\alpha_{k}^{2}-\alpha_{-k}^{2})\cos(k(t+s)).\nonumber 
\end{eqnarray}
{Let us consider the special case where $\Psi$ is a stationary random field.} This implies $\alpha_{k}^{2}=\alpha_{-k}^{2}$ and thus
\begin{equation*}
R(s,t)  =K(|t-s|),
\quad\text{and}\quad
\sigma_{s,t}^2 =2(K(0)-K(|t-s|))=:F(|t-s|),
\end{equation*}
where the function $K$ is defined by
\begin{eqnarray*}
K(t):=\sum_{k=1}^{\infty}\alpha_{k}^{2}\cos(kt).
\end{eqnarray*}
We now wish to prove that this situation can be seen as a particular case of Example \ref{example_stat}.
For simplicity we concentrate on the model-case
\begin{equation}\label{c1}
\alpha_{k}^{2}=Ck^{-(1+\frac{1}{\rho})}.
\end{equation}
for some $\rho \in [1,2)$, $C>0$. For more general conditions on the coefficients we refer to \cite[Section 3]{FGGR13}. By \cite[Section 3]{FGGR13}, $K$ is convex on $[0,2\pi]$, decreasing on $[0,\pi]$ and $\frac{1}{\rho}$-Hölder continuous. {In order to check the conditions of Example \ref{example_stat}}, it remains to verify the lower bound in \eqref{eq:assp_F}. We observe
\begin{align*}
F(t) =K(0)-K(t) & =\sum_{k=1}^{\infty}\alpha_{k}^{2}(1-\cos(kt))=2\sum_{k=1}^{\infty}\alpha_{k}^{2}\sin^{2}(\frac{kt}{2})\ge2\sum_{k=\lfloor\frac{1}{2t}\rfloor}^{\lfloor\frac{1}{t}\rfloor}\alpha_{k}^{2}\sin^{2}(\frac{kt}{2})\\
 & \gtrsim\sum_{k=\lfloor\frac{1}{2t}\rfloor}^{\lfloor\frac{1}{t}\rfloor}\alpha_{k}^{2}\gtrsim\alpha_{\lfloor\frac{1}{t}\rfloor}^{2}(\lfloor\frac{1}{t}\rfloor-\lfloor\frac{1}{2t}\rfloor)\gtrsim\alpha_{\lfloor\frac{1}{t}\rfloor}^{2}\lfloor\frac{1}{t}\rfloor\gtrsim t^{\frac{1}{\rho}},
\end{align*}
where we write $a \gtrsim b$ whenever $a\ge c \, b$ for a universal constant $c$ and where we have used inequality \eqref{c1} for the last step.
Since $F$ is not identically equal to zero, it follows that there is a time $T\in(0,2\pi]$, such that $F$ is concave, $\inf_{s\in[0,T]}F'_{-}(s)>0$, $F$ is $\frac{1}{\rho}$-Hölder continuous and \eqref{eq:assp_F} is satisfied. Hence, by Example \ref{example_stat} Hypothesis \ref{hyp:mixed-var-R}, \ref{hyp:correlation-increments-X} and~\ref{assumption1} are satisfied and $\eta$ is bounded on $[0,T]$. \end{example}

\begin{example}\label{exspectralmeasure} Let $X$ be a $d$-dimensional continuous, centred Gaussian process with i.i.d.~components. In the following $X_{t}$ denotes one of its components. Assume that $X_{t}$ is a stationary, zero-mean process with covariance 
\[
R(s,t)=K\bigl(\vert t-s\vert\bigr)
\]
for some continuous and positive definite function $K$. By Bochner's Theorem there is a finite positive symmetric measure $\mu$ on $\mathbb{R}$ such that
\begin{equation*}
K(t)  =  \int\cos(t\xi)\mu(d\xi)
\end{equation*}
and thus 
\begin{equation*}
\sigma^{2}(t) :=  \sigma^{2}_{0,t}=2\bigl(K(0)-K(t)\bigr)=4\int\sin^{2}(t\xi/2)\mu(d\xi).
\end{equation*}
The case of discrete $\mu$ corresponds to Example~\ref{exrfsnon-stationary}. Another example is given by the \textit{fractional Ornstein--Uhlenbeck process}, 
\[
X_{t}=\int_{-\infty}^{t}e^{-\lambda(t-u)}\,dB_{u}^{H},\qquad t\in\mathbb{R}.\label{eq:fracOU2}
\]
In this case, it is known that $X$ has a \emph{spectral density} $\mu(d\xi)$ such that
\begin{equation}
\frac{d\mu}{d\xi}=c_{H}\frac{\vert\xi\vert^{1-2H}}{\lambda^{2}+\xi^{2}}\equiv\hat{K}(\xi).\label{eq:fracmu-for-fOU}
\end{equation}
By Theorem~7.3.1 in \cite{MR06} we have that if $\hat{K}$ is regularly varying at $\infty$, then the coefficient $\si_{t}$ defined by \eqref{eq:def-variance-Xt} satisfies $\sigma_{t}^{2}\sim\frac{C\hat{K}(1/t)}{t}$ as $t\rightarrow0$ which in the case of \eqref{eq:fracmu-for-fOU} implies that there exists a $T>0$ such that
%Applied to the situation at hand we see that $\sigma^{2} (%t ) =O (%t^{2H} )$, since $\hat{K} ( \xi) \sim(%1/\xi) ^{1+2H}$. With focus on the rough case $H\leq1/2$, this gives%condition (B.iii) with $\rho=1/ ( 2H ) $, $\omega(s,t ) =t-s$. Applying this relation to the fractional Ornstein--Uhlenbeck case \eqref{eq:mu-for-fOU}, it is readily seen that there exists $\tau>0$ such that: 
\[
C_{1}t^{2H}\le\sigma^{2}(t)\le C_{2}t^{2H}\quad\text{for all }t\in[0,T].
\]
Moreover, it can be seen that there is a $T>0$ such that $K$ is convex on the interval $[0,T]$ (cf.~\cite[Example 5.3]{FGGR13}) and $\sup _{t\in [0,T]}K'(t)<0$. Hence, Hypothesis \ref{hyp:correlation-increments-X} and by \cite[equation (6.2)]{FGGR13} Hypothesis~\ref{assumption1} are satisfied. By \cite{FGGR13} we conclude %%which implies (B.i) and%(B.ii) as in Example~\ref{StInII}. Hence, it follows that 
\[
V_{1,\rho}(R;[s,t]^{2})=O(\vert t-s\vert^{2H})\text{ for all }[s,t]\subseteq[0,T].
\]
Hence, Hypothesis \ref{hyp:mixed-var-R} is satisfied and
\[
\eta_{t}\le C\quad\text{for all }t\in[0,T].
\]
\end{example}

\bigskip

\noindent{\bf Acknowledgment.}\quad We would like to thank the anonymous referee for his/her very careful reading of the first version of this paper and for many valuable suggestions.


\medskip
 \begin{thebibliography}{99}
 
\bibitem{BP}
V. Bally, E. Pardoux: 
Malliavin calculus for white noise driven parabolic SPDEs. 
{\it Potential Anal.} {\bf 9} (1998), no. 1, 27--64.

\bibitem{BH} F. Baudoin, M. Hairer: 
A version of H\"ormander's theorem for the fractional Brownian motion. 
{\it Probab. Theory Related Fields} {\bf 139} (2007), no. 3-4, 373--395.

\bibitem{BNOT} 
F. Baudoin, E. Nualart, C. Ouyang, S. Tindel: On probability laws of solutions to differential systems driven by fractional Brownian motion. {\it Ann. Probab.}, {\bf44} (2016), no. 4, 2554-25901.


\bibitem{BO} F. Baudoin, C. Ouyang: Small-time kernel expansion for solutions of stochastic differential equations driven by fractional Brownian motions. {\it Stoch. Proc. Appl.} {\bf 121} (2011), no. 4,  759--792.

\bibitem{BOT} F. Baudoin, C. Ouyang, S. Tindel: 
Upper bounds for the density of  solutions of stochastic differential equations driven by fractional Brownian motions.  {\it  Ann. Inst. Henri Poincar\'e Probab. Stat.}, {\bf50} (1), 2014, 111-135.

\bibitem{BM}
D. Bell, S.-E. Mohammed: 
The Malliavin calculus and stochastic delay equations. 
{\it J. Funct. Anal.} {\bf 99} (1991), no. 1, 75--99.

\bibitem{BO-Varadhan}F. Baudoin, C. Ouyang, X. Zhang: Varadhan estimates for RDEs driven by fractional Brownian motions. {\it Stochastic Processes and Their Applications}, Vol. 125, Issue 2, 634-652, (2015)


\bibitem{CF} T. Cass, P. Friz: Densities for rough differential equations under H\"{o}rmander's condition. {\it Ann. Math}. {\bf 171}, (2010), 2115--2141. 

\bibitem{CFV}
T. Cass, P. Friz, N. Victoir: 
Non-degeneracy of Wiener functionals arising from rough differential equations. 
{\it Trans. Amer. Math. Soc.} {\bf 361} (2009), no. 6, 3359--3371. 

\bibitem{CHLT} T. Cass, M. Hairer, C. Litterer, S. Tindel:
Smoothness of the density for solutions to Gaussian Rough Differential Equations.
{\it Ann. Probab.}, {\bf 43} (2015), no. 1, 188-239.


\bibitem{CLL} T. Cass, C. Litterer, T. Lyons:
Integrability and tail estimates for Gaussian rough differential equations.
{\it Ann. Probab} {\bf 41} (2013), no. 4, 3026--3050.

\bibitem{FGGR13}P.K. Friz, B. Gess, A. Gulisashvili, S. Riedel: Jain-Monrad criterion for rough paths. 
{\it Ann. Probab.}, {\it 44} (1),  684-738 (2016).


\bibitem{FR}P.K. Friz, S. Riedel: Integrability of (Non-)Linear Rough Differential Equations and Integrals. {\it Stochastic Analysis and Applications}, Vol 31, Issue 2, (2013).

\bibitem{FRS}
M. Ferrante, C. Rovira, M. Sanz-Solé: 
Stochastic delay equations with hereditary drift: estimates of the density. 
{\it J. Funct. Anal.} {\bf 177} (2000), no. 1, 138--177.

\bibitem{FH} 
P.K. Friz,  M. Hairer: 
\textit{A course on rough paths: with an introduction to regularity structures,} Universitext, Springer (2014). 

\bibitem{FV}
P.K. Friz, N. Victoir: 
Differential equations driven by Gaussian signals. 
Ann. Inst. Henri Poincaré Probab. Stat. 46 (2010), no. 2, 369--413.

\bibitem{FV-bk}
P.K. Friz, N. Victoir:
\emph{Multidimensional Stochastic Processes as Rough Paths.}
Cambridge University Press (2010).

\bibitem{Gu}
M. Gubinelli:
Controlling rough paths.
{\it J. Funct. Anal.} {\bf 216}, 86-140 (2004).

\bibitem{H}M. Hairer: Ergodicity of stochastic differential equations driven by fractional Brownian motion. Ann. Probab., {\bf 33}, No. 2 (2005), 703-758.

\bibitem{HV03}
C. Houdr{\'e}, J. Villa: \emph{An example of infinite
  dimensional quasi-helix}, Stochastic models ({M}exico {C}ity, 2002), Contemp.
  Math., vol. 336, Amer. Math. Soc., Providence, RI, 2003, pp.~195--201.

\bibitem{Inahama} Y. Inahama: Malliavin differentiability of solutions of rough differential equations. {\it J. Funct. Anal.}, {\bf267} (5), 1566-1584 (2013). 

\bibitem{Inahama2} Y. Inahama: Short time kernel asymptotics for rough differential equation driven by fractional Brownian motion. Eletron. J. Probab., {\bf21} (2016), no. 34, 1-29.

\bibitem{KRT07}
I. Kruk, F. Russo, C.A. Tudor: \emph{Wiener integrals,
  {M}alliavin calculus and covariance measure structure}, J. Funct. Anal.
  \textbf{249} (2007), no.~1, 92--142.
  

\bibitem{Leandre PTRF87}R. L\'{e}andre: Integration dans la fibre associ\'{e}e a une diffusion d\'{e}g\'{e}n\'{e}r\'{e}e. {\it Probab. Theory Related Fields}, Vol 76, Issue 3, 341-358, (1987).

\bibitem{LO2} S. Lou and C. Ouyang: Local times of stochastic differential equations driven by fractional Brownian motions. {\it Stochastic Process Appl.}, Vol. 127, Issue 11, 3643-3660, (2017).


\bibitem{Ma}
P. Malliavin:
Stochastic calculus of variation and hypoelliptic operators. 
Proceedings of the International Symposium on Stochastic Differential Equations, pp. 195--263. 
Wiley, 1978.

\bibitem{MR06}
M.B. Marcus, J. Rosen: \emph{Markov processes, {G}aussian processes,
  and local times}, Cambridge Studies in Advanced Mathematics, vol. 100,
  Cambridge University Press, Cambridge, 2006.


\bibitem{Nu06}
D. Nualart: \emph{The Malliavin Calculus and Related
Topics.} Probability and its Applications. Springer-Verlag, 2nd
Edition, (2006).


\bibitem{NQ}
D. Nualart, L. Quer-Sardanyons: 
Existence and smoothness of the density for spatially homogeneous SPDEs. 
{\it Potential Anal.} {\bf 27} (2007), no. 3, 281--299.

\bibitem{NS}
D. Nualart, B. Saussereau:
Malliavin calculus for stochastic differential equations driven by a fractional  Brownian motion.  {\it Stochastic Process. Appl.}  {\bf 119}  (2009),  no. 2, 391--409. 


\bibitem{RS}
C. Rovira, M. Sanz-Solé: 
The law of the solution to a nonlinear hyperbolic SPDE. 
{\it J. Theoret. Probab.} {\bf 9} (1996), no. 4, 863--901.

\bibitem{RT06}
  F. Russo, C.A. Tudor: \emph{On bifractional {B}rownian motion},
    Stochastic Process. Appl. \textbf{116} (2006), no.~5, 830--856.

\bibitem{Tow02}
N. Towghi: 
Multidimensional extension of L. C. Young's inequality, 
\emph{J. Inequal. Pure Appl. Math.} {\bf 3} (2002), no. 2, Article 22, 13 pp. (electronic).


\end{thebibliography}
\end{document}